\documentclass[12pt]{article}

\usepackage{amssymb,amsmath,latexsym}

\usepackage{authblk}
\usepackage{amssymb,amscd,amsmath,amsfonts}
\usepackage{latexsym}
\usepackage{amsthm}
\usepackage{enumerate}
\usepackage{verbatim}
\usepackage{graphicx}
\usepackage[dvips]{epsfig}
\usepackage{subcaption}
\allowdisplaybreaks

\newtheorem{theorem}{Theorem}[section]
\newtheorem{lemma}[theorem]{Lemma}
\newtheorem{proposition}[theorem]{Proposition}
\newtheorem{definition}[theorem]{Definition}

\newtheorem{corollary}[theorem]{Corollary}
\newtheorem{remark}[theorem]{Remark}

\newtheorem{thm}{Theorem}[section]

\newtheorem{rk}[thm]{Remark}

\numberwithin{equation}{section}

\newcommand{\C}{\mathbb C}
\newcommand{\R}{\mathbb R}
\newcommand{\Z}{\mathbb Z}
\newcommand{\N}{\mathbb N}

\newcommand{\RR}{\mathcal R}

\newcommand{\DD}{\mathcal D}

\newcommand{\re}{\mathrm{Re}}
\newcommand{\im}{\mathrm{Im}}
\newcommand{\e}{\varepsilon}
\newcommand{\Orb}{\operatorname{Orb}}

\begin{document}

\title{Hyers-Ulam stability of loxodromic M\"obius difference equation}

\author
{Young Woo Nam}

\affil[]{\small Mathematics Section,
      College of Science and Technology,
      Hongik University, 339--701 
      Sejong, Korea}

\date{}

\maketitle

\begin{abstract}
Hyers-Ulam of the sequence $ \{z_n\}_{n \in \mathbb{N}} $ satisfying the difference equation $ z_{i+1} = g(z_i) $ where $ g(z) = \frac{az + b}{cz + d} $ with complex numbers $ a $, $ b $, $ c $ and $ d $ is defined. Let $ g $ be {\em loxodromic} M\"obius map, that is, $ g $ satisfies that $ ad-bc =1 $ and $a + d \in \mathbb{C} \setminus [-2,2] $. Hyers-Ulam stability holds if the initial point of $ \{z_n\}_{n \in \mathbb{N}} $ is in the exterior of {\em avoided region}, which is the union of the certain disks of $ g^{-n}(\infty) $ for all $ n \in \mathbb{N} $. 
\end{abstract}

\section{Introduction}
%
Ulam \cite{ulam} posed the the stability of group homomorphisms in 1940. Given a metric group $(G, \cdot, d)$ for given $ \varepsilon > 0 $, suppose that a function $f : G \to G$ which satisfies the inequality $d \big( f(xy),\; f(x)f(y) \big) \leq \varepsilon$ for all $x, y \in G$. The question is about the  existence of a homomorphism $a : G \to G$ such that $d \big( a(x), f(x) \big) \leq \delta$ where $\delta$ depends only on $G$ and $\varepsilon$ for all $x \in G$. 
Hyers \cite{hyers} gave the affirmative answer this question in 1941 for Cauchy additive equation in Banach spaces. Hyers-Ulam stability has been developed for functional equations for a few decades by various authors. For a few example, see \cite{BCL, hyers, hir, rassias1}. 
\smallskip \\
The difference equation has Hyers-Ulam stability if each terms of the sequence with the given relation has (small) error, this sequence is approximated by the sequence with same relation which has no error. For the introduction of difference equation, for example, see \cite{elaydi}. Hyers-Ulam stability of difference equations, see \cite{jung1, jungnam, jungnam1, nam1, PoADE}. Especially Pielou logistic difference equation has Hyers-Ulam stability in \cite{jungnam1} only if the initial point of the sequence is contained in definite intervals. 
Hyers-Ulam stability is extended on the complex plane in \cite{nam2} the difference equation as follows 
\begin{align*}
z_{i+1} = \frac{az_i + b}{cz_i + d}
\end{align*}
over $ \C $ where $ ad - bc = 1 $, $ c \neq 0 $ and $ a + d \in \R \setminus [-2,2] $ for complex numbers $ a $, $ b $, $ c $ and $ d $. In this paper, we generalize the result the case that $ a+d \in \C \setminus [-2,2] $, that is, the map $ z \mapsto \frac{az+b}{cz+d} $ is the {\em loxodromic} M\"obius map. 

\subsection*{M\"obius map}
Linear fractional map on the Riemann sphere $ \hat{\C} = \C \cup \{\infty\} $ is called {\em M\"obius map} or {\em M\"obius transformation} 
$$ g(z) = \dfrac{az + b}{cz + d} $$ 
where $ ad - bc \neq 0 $ for $ z \in \hat{\C} $.
\smallskip \\
The non-constant M\"obius map $ g(z) = \frac{az+b}{cz+d} $ has the following properties.
\begin{itemize}
\item Without loss of generality, we may assume that $ ad - bc =1 $.
\item $ g(\infty) $ is defined as $ \frac{a}{c} $ and $ g \left(-\frac{d}{c} \right) $ is defined as $ \infty $.
\item The composition of two M\"obius maps is also a M\"obius map.
\item The map $ g $ is the linear map if and only if $ \infty $ is a fixed point of $ g $.
\item The image of circle or line under M\"obius map is circle or line.
\end{itemize}
%
The equation $ \frac{az+b}{cz+d} = \frac{paz + pb}{pcz +pd} $ holds for all $ p \neq 0 $. We define the 
%
matrix representation of M\"obius map $ z \mapsto \frac{az+b}{cz+d} $ as follows
$ \left( 
\begin{smallmatrix}
a & b \\
c & d
\end{smallmatrix}
\right) $
where $ ad - bc = 1 $. 
Thus we assume that $ ad-bc = 1 $ throughout this paper. Denote the trace of the matrix representation of M\"obius map $ g $ by $ \mathrm{tr}(g) $ and we call $ \mathrm{tr}(g) $ the trace of $ g $. 
%
%




\section{M\"obius map with attracting and repelling fixed points}
 
The trace of matrix is invariant under conjugation. Thus qualitative classification of M\"obius map depends on the trace of matrix representation. For this classification, see \cite{Beardon}.
\begin{definition}
If the matrix representation of the non-constant M\"obius map  
$ \left( 
\begin{smallmatrix}
a & b \\
c & d
\end{smallmatrix}
\right) $
has its trace $ a+d $, say $ \mathrm{tr}(g) $, is in the set $ \C \setminus [-2, 2] $, then the map $ g $ is called the {\em loxodromic} M\"obius map. Moreover, if $ \mathrm{tr}(g) $ is $ \C \setminus \R $, then map $ g $ is called {\em purely loxodromic}. 
\end{definition}
\noindent Denote the fixed points of $ g $ by $ \alpha $ and $ \beta $. If $ |g'(\alpha)| < 1 $, then $ \alpha $ is called the {\em attracting} fixed point. If $ |g'(\beta)| > 1 $, then $ \beta $ is called the {\em repelling} fixed point.


\begin{lemma} \label{lem-fixed points of loxo Mobius map}
Let $ g $ be the M\"obius map such that $ g(z) = \frac{az +b}{cz +d} $ where $ ad - bc =1 $ and $ c \neq 0 $. If $ g $ is loxodromic, then $ g $ has two different fixed points, one of which is the attracting fixed point and the other is the repelling fixed point.
\end{lemma}

\begin{proof}
The fixed points of $ g $ are the roots of the quadratic equation
$$ cz^2 - (a-d)z - b =0 . $$
Denote the fixed points of $ g $ as follows
\begin{align} \label{eq-fixed points of g}
\alpha = \frac{a-d + \sqrt{(a+d)^2 - 4}}{2c} \quad \textrm{and} \quad \beta = \frac{a-d - \sqrt{(a+d)^2 - 4}}{2c}. 
\end{align}
Observe that $ \alpha + \beta = \frac{a-d}{c} $ and $ \alpha \beta = -\frac{b}{c} $. Thus we have the following equation
\begin{align} \label{eq-ad-bc is 1}
 \nonumber
(c\alpha +d)(c\beta +d) &= c^2\alpha\beta + cd (\alpha + \beta ) + d^2 \\ \nonumber
&= -bc + d(a-d) +d^2 \\ \nonumber
&= -bc +ad \\
&=1
\end{align}
Thus we obtain the following inequality using the equation \eqref{eq-fixed points of g}
\begin{align} \label{eq-c alpha plus d}
c\alpha + d = \frac{a+d + \sqrt{(a+d)^2 - 4}}{2}.
\end{align}
We claim that $ | c\alpha + d | \neq 1 $. Denote temporarily $ a+d $ by $ z $ and $ c\alpha + d $ by $ w $. Then we have
$ w = \frac{z + \sqrt{z^2-4}}{2}. $ The straightforward calculation shows that 
$$ z = w + \frac{1}{w} . $$ 
Suppose that $ |w| = 1 $. Then
\begin{align*}
\mathrm{tr}(g) = a+d = z = w + \frac{1}{w} = e^{i\theta} + e^{-i\theta} = 2\cos \theta.
\end{align*}
Then $ \mathrm{tr}(g) $ is the real number and $ |\mathrm{tr}(g)| \leq 2 $. It contradicts to the fact that $ g $ is the loxodromic M\"obius map. Then we may assume that $ | c\alpha + d | > 1 $. Since $ g'(z) = \frac{1}{(cz+d)^2} $ and by the equations \eqref{eq-ad-bc is 1} and \eqref{eq-c alpha plus d}, we obtain that $ | g'(\alpha) | = \frac{1}{ | c\alpha+d |^2} < 1 $ and $ | g'(\beta) | = \frac{1}{ | c\beta+d |^2} > 1 $. Hence, $ g $ has both attracting and repelling fixed points. 
\end{proof}

\smallskip

\begin{lemma} \label{lem-congugation of loxo Mobius map}
Let $ g $ and $ h $ are M\"obius map as follows
$$
g(z) = \frac{az+b}{cz+d} \quad  \text{and} \quad h(z) = \frac{z-\beta}{z-\alpha}
$$
where $ \alpha $ and $ \beta $ are the fixed points of $ g $ and $ ad-bc =1 $. If $ \alpha \neq \beta $, then $ h \circ g \circ h^{-1} (w) = kw $ where $ k = \frac{1}{(c\beta +d)^2} $. In particular, if $ g $ is the loxodromic M\"obius map and $ \beta $ is the repelling fixed point, then $ |k|>1 $. 
\end{lemma}

\begin{proof}
The maps $ g $ and $ h $ are M\"obius map. Thus so is $ h \circ g \circ h^{-1} $. By the direct calculation, we obtain that $ h^{-1}(w) = \frac{\alpha w - \beta}{w - 1} $.  Observe that $ h^{-1}(0) = \beta $, $ h^{-1}(\infty) = \alpha $ and $ h^{-1}(1) = \infty $. Then we have
\begin{align*}
h \circ g \circ h^{-1}(0) &= h \circ g (\beta) = h(\beta) = 0 \\
h \circ g \circ h^{-1}(\infty) &= h \circ g (\alpha) = h(\alpha) = \infty
\end{align*}
The points $ 0 $ and $ \infty $ are fixed points of $ h \circ g \circ h^{-1} $. So $ h \circ g \circ h^{-1} (w) = kw $ for some $ k \in \C $. Since $ k = h \circ g \circ h^{-1} (1) $, the following equation holds by \eqref{eq-fixed points of g} and \eqref{eq-ad-bc is 1}
\begin{align*}
k &= h \circ g \circ h^{-1} (1) = h \circ g(\infty) = h \left( \frac{a}{c} \right) & \\
&= \frac{\frac{a}{c} - \beta}{\frac{a}{c} - \alpha} = \frac{a - c\beta}{a - c\alpha} & 
\\
&= \frac{a + d + \sqrt{(a+d)^2 - 4}}{a + d - \sqrt{(a+d)^2 - 4}}  
\\
&= \frac{c\alpha +d}{c\beta + d} 
\\
&= \frac{1}{(c\beta +d)^2}.  
\end{align*}
%
If $ g $ is the loxodromic M\"obius map, then $ |k| = \frac{1}{|c\beta +d|^2} > 1 $ by the proof of Lemma \ref{lem-fixed points of loxo Mobius map}. 
\end{proof}

\smallskip

\begin{corollary}
In Lemma \ref{lem-congugation of loxo Mobius map}, if $ g $ is the purely loxodromic map, then $ k $ cannot be a positive real number. 
\end{corollary}

\begin{proof}
The proof of Lemma \ref{lem-fixed points of loxo Mobius map} implies that $ \mathrm{tr}(g) = w + \frac{1}{w} $ where $ w = c\alpha + d $ is the complex number and $ |w| \neq 1 $. Denote $ w $ by $ re^{i\theta} $ for some $ r \neq 1 $. Then the trace of $ g $ satisfies that
$$ \mathrm{tr}(g) = \left(r + \frac{1}{r} \right) \cos \theta + i \left(r - \frac{1}{r} \right) \sin \theta . $$
Since $ g $ is purely loxodromic, we have $ \mathrm{tr}(g) \in \C \setminus \R $, that is, $ \theta \neq 0 $. Then $ c\alpha + d $ is a non real complex number. Hence, $ k = (c\alpha + d)^2 $ cannot be any positive real number. 
\end{proof}

\smallskip

\begin{lemma} \label{lem-loxodromic Mobius transformation}
Let $ g $ be the loxodromic M\"obius map on $ \hat{\C} $. Let $ \alpha $ and $ \beta $ be the attracting and the repelling fixed point respectively. Then 
$$
\lim_{n \rightarrow \infty} g^n(z) \rightarrow \alpha \quad \text{as} \quad n \rightarrow +\infty
$$ 
for all $ z \in \hat{\C} \setminus \{\beta \} $. 
\end{lemma}

\begin{proof}
By the classification of M\"obius map, loxodromic M\"obius map has both the attracting and the repelling fixed points. Let $ h $ be the linear fractional map as follows
$$
h(z) = \frac{z - \beta}{z - \alpha} .
$$
Then $ f = h \circ g \circ h^{-1} $ is the dilation with the repelling fixed point at zero, that is, $ f(w) = kw $ for $ |k| > 1 $. Thus  $ 0 $ is the repelling fixed point of $ f $. Since $ h $ is a bijection on $ \hat{\C} $, the orbit, $ \{ g^n(z) \}_{n \in \Z} $ corresponds to the orbit, $ \{ f^n(h(z)) \}_{n \in \Z} $ by conjugation $ h $. Observe that 
$$
f^n(z) \rightarrow \infty \quad \text{as} \quad n \rightarrow +\infty
$$
for all $ z \in \hat{\C} \setminus \{ 0 \} $. Hence, 
$$
g^n(z) \rightarrow \alpha \quad \text{as} \quad n \rightarrow +\infty
$$
for all $ z \in \hat{\C} \setminus \{ \beta \} $. 
\end{proof}

\smallskip

\begin{corollary}
Let $ g $ be the map defined in Lemma \ref{lem-loxodromic Mobius transformation}. Then 
$$
\lim_{n \rightarrow \infty} g^{-n}(z) \rightarrow \beta \quad \text{as} \quad n \rightarrow +\infty
$$ 
for all $ z \in \hat{\C} \setminus \{\alpha \} $.
\end{corollary}

\begin{proof}
Observe that $ g^{-1} $ is also loxodromic M\"obius map and $ \beta $ and $ \alpha $ are the attracting and the repelling fixed point under $ g^{-1} $ respectively. Thus we apply the proof of Lemma \ref{lem-loxodromic Mobius transformation} to the map $ g^{-1} $. It completes the proof.  
\end{proof}
\smallskip
We collect the notions throughout this paper as follows
\begin{itemize}
\item The map $ g $ is the purely loxodromic M\"obius map and $ g(z) = \frac{az + b}{cz +d} $ where $ ad -bc =1 $ and $ c \neq 0 $.
\item The M\"obius map $ h $ is defined as $ h(z) = \frac{z - \beta}{z - \alpha} $ where $ \alpha $ and $ \beta $ are the attracting and the repelling fixed points of $ g $.
\item Without loss of generality, we may assume that the purely loxodromic M\"obius map $ g $ has the matrix representation where $ \mathrm{tr}(g) $ is the non real complex number. 
\item Since the trace of matrix is invariant under conjugation, we obtain that $ \mathrm{tr}(g) = \mathrm{tr}(h \circ g \circ h^{-1}) $. Then  
\begin{align*}
\mathrm{tr}(g) = \mathrm{tr}
\begin{pmatrix}
\sqrt{k} & 0 \\ 
0 & \frac{1}{\sqrt{k}}
\end{pmatrix} 
 = \sqrt{k} + \frac{1}{\sqrt{k}} .
\end{align*}
\end{itemize}

\smallskip

\section{Image of circles under the conjugation}
In this section, we show that the image of circles under the map $ h $ defined as $ h(z) = \frac{z- \beta}{z- \alpha} $.  Recall that the image of line or circle under M\"obius map is line or circle.  
The map $ f = h \circ g \circ h^{-1} $ is the dilation defined as $ f(w) = kw $ where $ k = \frac{1}{(c \beta + d)^2}  $ and $ |k| > 1 $ by Lemma \ref{lem-congugation of loxo Mobius map}. 
%
%
%
%
%
\smallskip
\begin{lemma} \label{lem-image of -d/c under h}
Let $ h $ be the M\"obius map defined as $ h(z) = \frac{z- \beta}{z- \alpha} $. Then the image of $ -\frac{d}{c} $ under $ h $ as follows 
$$ h\left(-\frac{d}{c} \right) = \frac{1}{k}, \quad  -\frac{d}{c} = \frac{k \beta - \alpha}{k-1}  \ \  \textrm{and} \ \ [\;\!c(\alpha - \beta)]^2 = \frac{(k-1)^2}{k}. $$ 
\end{lemma}

\begin{proof}
The map $ h $ is the conjugation from $ g $ to $ f $ and $ h(\infty) = 1 $. The fact that $ f \circ h = h \circ g $ implies that
$$
f \circ h\left(-\frac{d}{c} \right) = h \circ g \left(-\frac{d}{c} \right) = h(\infty) = 1 = f \left(\frac{1}{k} \right) . 
$$
Since $ f $ is a bijection on $ \C $, $ h(-\frac{d}{c}) = \frac{1}{k} $. Observe that the map $ h^{-1}(w) = \frac{\alpha w - \beta}{w - 1} $. Hence, we have 
\begin{align} \label{eq-inverse image of -d/c}
 h^{-1}\left( \dfrac{1}{k} \right) = -\dfrac{d}{c} 
 = \alpha - \frac{k}{k-1}(\alpha - \beta) .
\end{align}
The equation \eqref{eq-inverse image of -d/c} implies that $ c\alpha + d = \frac{k}{k-1}c(\alpha - \beta) $. Since $ k = (c\alpha + d)^2 $ by \eqref{eq-ad-bc is 1} and Lemma \ref{lem-congugation of loxo Mobius map}, we have 
$$ [\;\!c(\alpha - \beta)]^2 = \frac{(k-1)^2}{k} . $$
\end{proof}

\noindent Let 
the circle $ \{ z \colon \ |z- \beta | = r | z- \alpha | \} $  be $ C(r) $ for $ r > 0 $. In particular, denote $ C(1) \cup \{ \infty \} $ by $ L_{\infty} $. Similarly, denote the region $ \{ z \colon \ |z- \beta | \geq r | z- \alpha | \} $ by $ B(r) $ for $ r \geq 0 $. Observe that $ r_1 \geq r_2 $ if and only if $ B(r_1) \subset B(r_2) $.

\begin{lemma} \label{lem-image of circle under h}
Let $ g $ be the M\"obius map with two different fixed points $ \alpha $ and $ \beta $. Let $ h $ be the map $ h(z) = \frac{z-\beta}{z-\alpha} $. Then $ h(C(r)) = \{w \colon \ |w| = r \} $ for $ r \neq 1 $. In particular, $ h(L_{\infty}) = \{w \colon \ |w| = 1 \} $.
\end{lemma}

\begin{proof}
Let $ w = h(z) $. Then
\begin{align} \label{eq-image of circle to circle1}
|w| = \left| \frac{z-\beta}{z-\alpha} \right| =  \frac{|r (z-\alpha) |}{| z-\alpha |} = |r| = r
\end{align}
for $ r > 0 $. Since M\"obius transformation is bijective on $ \hat{\C} $, we have that $ h(C(r)) = \{w \colon \ |w| = r \} $ for $ r \neq 1 $. In $ r = 1 $ case, $ C(1) $ is the straight line rather than geometric circle and $ h(\infty) = 1 $. Then by the equation \eqref{eq-image of circle to circle1}, $ h(L_{\infty}) = \{w \colon \ |w| = 1 \} $. 
\end{proof}

\noindent Denote $ \{ w \colon \ |w| \geq r \} $ by $ D(r) $ for $ r \geq 0 $.

\begin{lemma}
Let $ g $ be the loxodromic map. Then $ B(r) $ is invariant under $ g $, that is, $ g(B(r)) \subset B(r) $. Furthermore, $ g(B(r)) = B(|k|r) $ where $ k = \frac{1}{(c\beta + d)^2} $. 
\end{lemma}

\begin{proof}
The map $ f = h \circ g \circ h^{-1} $ where $ f(w) = kw $ by Lemma \ref{lem-congugation of loxo Mobius map}. Observe that $ f(D(r)) = D(|k|r) $. Moreover, $ h(B(r)) = D(r) $ by the similar proof of Lemma \ref{lem-image of circle under h}. Thus
\begin{align*}
g(B(r)) = h^{-1} \circ f \circ h(B(r)) = h^{-1} \circ f(D(r)) = h^{-1}(D(|k|r)) = B(|k|r) .
\end{align*}
Since $ |k| > 1 $, $ B(|k|r) \subset B(r) $. Hence, $ B(r) $ is invariant under $ g $. 
\end{proof}
\bigskip
\noindent Define the region $ S(r) $ as follows
\begin{align}  
S(r) = \left\{z \in \C \colon \left| z + \frac{d}{c} \right| > \frac{r}{|c|} \right\} \label{eq-definition of Sr}
\end{align}
for $ r > 0 $. 

\begin{proposition} \label{prop-image of circle with center -d/c}
Let $ g $ be the loxodromic M\"obius map and $ h $ be the map defined as $ h(z) = \frac{z- \beta}{z - \alpha} $. Then
$$
h(S(r)) = \left\{ w \colon \ \frac{\sqrt{|k|}}{r} \left| w - \frac{1}{k} \right| > |w-1| \right\}
$$
for $ r > 0 $. 
\end{proposition}

\begin{proof}
The definition of $ S(r) $ shows that $ z \in S(r) $ if and only if $ |cz + d| > r $. Lemma \ref{lem-image of -d/c under h} implies that 
$ c\alpha +d = \frac{k}{k-1}\,c(\alpha - \beta) $ and $ \frac{(k-1)^2}{k} = [c(\alpha - \beta)]^2 $. Moreover, $ w = \frac{z- \beta}{z - \alpha} $ if and only if $ z = \alpha + \frac{\alpha - \beta}{w-1} $. Thus
\begin{align*}
cz + d &= c\alpha + \frac{c(\alpha - \beta)}{w-1} + d \\
&= \frac{k}{k-1}\,c(\alpha - \beta) + \frac{c(\alpha - \beta)}{w-1} \\
&= c(\alpha - \beta) \left[ \frac{k}{k-1} + \frac{1}{w-1} \right] \\
&= \frac{c(\alpha - \beta)}{k-1} \cdot \frac{k(w-1) + k-1}{w-1}  \\
&= \frac{c(\alpha - \beta)}{k-1} \cdot \frac{kw-1}{w-1} 
\end{align*}
Then
\begin{align*}
|cz+d| = \left| \frac{c(\alpha - \beta)}{k-1} \right| \cdot \left| \frac{kw-1}{w-1} \right| = \frac{|k-1| / \sqrt{|k|}}{|k-1|} \cdot  \frac{|k(w-\frac{1}{k})|}{|w-1|} = \sqrt{|k|} \,\frac{|w-\frac{1}{k}|}{|w-1|}
\end{align*}
Then $ |cz + d| > r $ implies that 
\begin{align*}
 \sqrt{|k|} \,\frac{|w-\frac{1}{k}|}{|w-1|} > r \
& \Longleftrightarrow \  \sqrt{|k|} \left| w-\frac{1}{k} \right| > r|w-1| \\
& \Longleftrightarrow \  \frac{\sqrt{|k|}}{r} \left| w-\frac{1}{k} \right| > |w-1| .
\end{align*}
Hence, $$
h(S(r)) = \left\{ w \colon \ \frac{\sqrt{|k|}}{r} \left| w - \frac{1}{k} \right| > |w-1| \right\}
$$
for $ r > 0 $. 
\end{proof}

\medskip

\begin{figure}
    \centering
    \begin{subfigure}[b]{0.45\textwidth}
        \includegraphics[width=\textwidth]{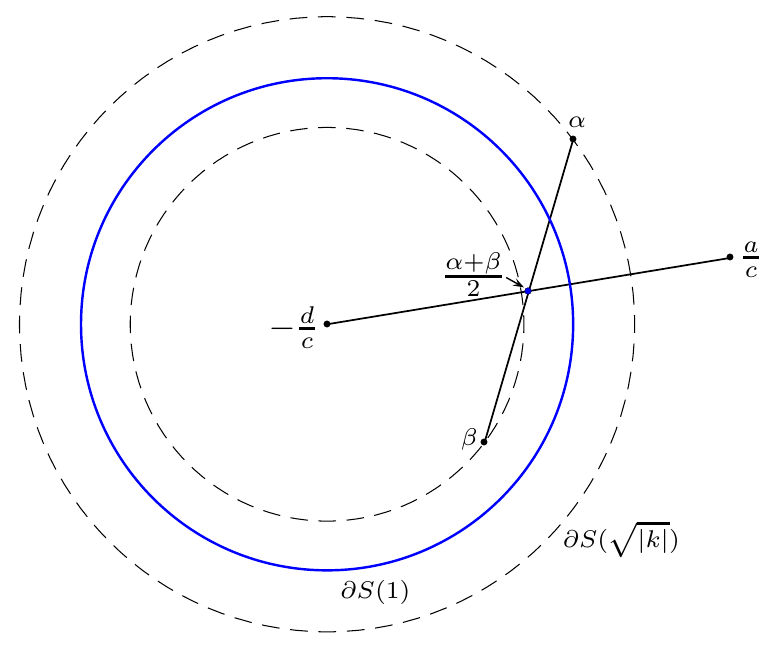}
        \caption{$S(1)$ and other circles}
        \label{fig:S1 and other circles}
    \end{subfigure} \
    ~ 
    \begin{subfigure}[b]{0.45\textwidth}
        \includegraphics[width=\textwidth]{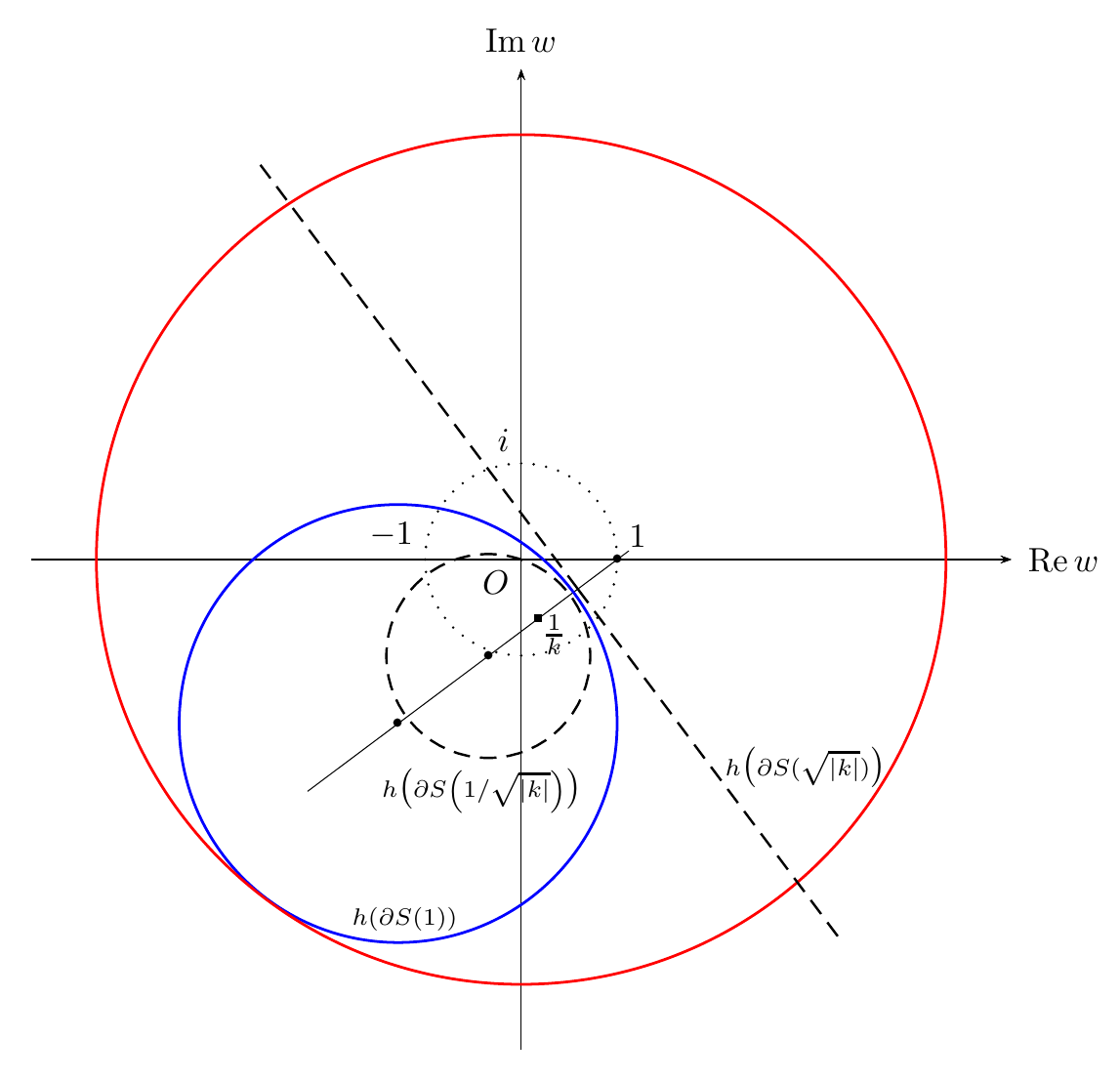}
        \caption{Images of $ S(r) $ under $ h $}
        \label{fig:image of Sr under h}
    \end{subfigure}
    \caption{$ S(r) $ and its image under $ h $}\label{fig:circles and images under h}
\end{figure}

\noindent Denote the boundary of the set $ S $ by $ \partial S $ and the closure of the set $ S $ by $ \overline{S} $. 

\begin{lemma} \label{lem-center and radius of circle}
The center of the circle $ h(\partial S(r)) $ is $ \frac{kr^2 - |k|}{k(r^2 - |k|)} $. The radius of $ h(\partial S(r)) $ is $ \frac{r|k-1|}{\sqrt{|k|}\,\big|r^2 - |k| \big|} $ for $ r > 0 $. 
\end{lemma}

\begin{proof}
$ \partial S(r) $ are the concentric circles of which center is $ -\frac{d}{c} $ for all $ r > 0 $. Thus any straight line which contains $ -\frac{d}{c} $ meets all circles $ \partial S(r) $ perpendicularly at the ends points of diameter of every circles. Let $ \ell_C $ the straight line which contains $ -\frac{d}{c} $ and $ \alpha $, that is, $ \ell_S = \left\{ z \colon \   z = s\alpha - (1-s)\frac{d}{c}, \ s \in \R \right\} $. Since $ h $ is conformal, $ h(\ell_S) $ and $ h(\partial S(r)) $ also meets orthogonally each other for all $ r > 0 $. Then $ h(\ell_S) \cap h(\partial S(r)) $ is the set of end points of the diameter of $ h(\partial S(r)) $ for each $ r > 0 $. Then the middle point of two points in $ h(\ell_S) \cap h(\partial S(r)) $ is the center of the circle $ \partial S(r) $ and the half of the distance between two points in $ h(\ell_S) \cap h(\partial S(r)) $ is the radius of $ \partial S(r) $. \\
Since $ h \left(-\frac{d}{c} \right) = \frac{1}{k} $, $ h(\alpha) = \infty $ and $ h(\infty) = 1 $, $ h(\ell_S) \cup \{ \infty \} $ is the extended straight line which contains $ \frac{1}{k} $ and $ 1 $, that is, 
$$ h(\ell_S) \cup \{ \infty \} = \left\{ w \colon \ w = t + (1-t) \frac{1}{k}, \ t \in \R \right\} \cup \{ \infty \} . $$ 
Proposition \ref{prop-image of circle with center -d/c} implies that
$$
h(\partial S(r)) = \left\{ w \colon \ \frac{\sqrt{|k|}}{r} \left| w - \frac{1}{k} \right| = |w-1| \right\} .
$$
Solve the equation for $ t $ as follows to determine the points in $ h(\ell_S) \cap h(\partial S(r)) $. 
\begin{align*}
\frac{\sqrt{|k|}}{r} \left| t + (1-t) \frac{1}{k} - \frac{1}{k} \right| &= \left| t + (1-t) \frac{1}{k}-1 \right| \\
\frac{\sqrt{|k|}}{r} |t| \left| 1 - \frac{1}{k} \right| &= |t-1| \left| 1 - \frac{1}{k} \right| \\
\frac{\sqrt{|k|}}{r} |t| &= |t-1| .
\end{align*}
Thus the values of $ t $ for the above equation are $ \frac{1}{1-\frac{\sqrt{|k|}}{r}} $ or $ \frac{1}{1+\frac{\sqrt{|k|}}{r}} $. Then the end points of the diameter of $ h(\partial S(r)) $ are as follows 
\begin{align*}
\frac{r}{r - \sqrt{|k|}} + \left( 1- \frac{r}{r - \sqrt{|k|}} \right) \frac{1}{k} \quad \text{and} \quad \frac{r}{r + \sqrt{|k|}} + \left( 1- \frac{r}{r + \sqrt{|k|}} \right) \frac{1}{k} .
\end{align*}
The center of the circle $ h(\partial S(r)) $ is the arithmetic average of two end points as follows
\begin{align*}
\quad  &\frac{1}{2} \left\{ \frac{r}{r - \sqrt{|k|}} + \left( 1- \frac{r}{r - \sqrt{|k|}} \right) \frac{1}{k} + \frac{r}{r + \sqrt{|k|}} + \left( 1- \frac{r}{r + \sqrt{|k|}} \right) \frac{1}{k} \right\} \\[0.2em]
= \ & \frac{1}{2} \left\{ \frac{r}{r - \sqrt{|k|}} \left(1 - \frac{1}{k} \right) + \frac{r}{r + \sqrt{|k|}} \left(1 - \frac{1}{k} \right) + \frac{2}{k} \right\} \\[0.2em]
= \ & \frac{1}{2} \left\{ \frac{2r^2}{r^2 - |k|} \left(1 - \frac{1}{k} \right) + \frac{2}{k} \right\} \\[0.2em]
= \ & \frac{r^2}{r^2 - |k|} \cdot \frac{k-1}{k} + \frac{1}{k} \\[0.2em]
= \ & \frac{kr^2 - |k|}{k(r^2 - |k|)} .
\end{align*}
The radius of the circle $ h(\partial S(r)) $ is the half of the distance between two end points of the diameter 
\begin{align*}
\quad  &\frac{1}{2} \left| \frac{r}{r - \sqrt{|k|}} + \left( 1- \frac{r}{r - \sqrt{|k|}} \right) \frac{1}{k} - \left\{ \frac{r}{r + \sqrt{|k|}} + \left( 1- \frac{r}{r + \sqrt{|k|}} \right) \frac{1}{k} \right\} \right| \\[0.2em]
= \ & \frac{1}{2} \left| \frac{r}{r - \sqrt{|k|}} \left(1 - \frac{1}{k} \right) - \frac{r}{r + \sqrt{|k|}} \left(1 - \frac{1}{k} \right) \right| \\[0.2em]
= \ & \left| \frac{r\sqrt{|k|}}{r^2 - |k|} \cdot \frac{k-1}{k} \right| \\[0.2em]
= \ & \frac{r|k-1|}{\sqrt{|k|} \cdot \big|r^2 - |k| \big|} .
\end{align*}
\end{proof}

\begin{rk}
If $ r= \sqrt{|k|} $, then $  h(\partial S(r)) $ is the straight line, which is perpendicular to the line segment between $ \frac{1}{k} $ and $ 1 $ and it contains $ \frac{1}{2}\left(1+ \frac{1}{k} \right) $. The circle $ h(\partial S(r)) $ contains the origin if and only if $ r= {1}/{\sqrt{|k|}} $.   The modulus of the center of $ h \left(\partial S \left({1}/{\sqrt{|k|}} \right) \right) $ is the same as the radius of the same circle. In this case, the radius is $ \frac{|k-1|}{|k|^2-1} $. 
\end{rk}

\medskip

\begin{corollary} \label{cor-h(partial S1)}
The region $ h(S(1)) $ contains the following region  
$$
\left\{ w \colon \ |w| > \frac{\sqrt{|k|} \, |k-1| + |k - |k||}{|k|(|k|-1)} \right\} 
$$
and the bounds of radius is as follows
\begin{align*}
\frac{1}{\sqrt{|k|}} \leq \frac{\sqrt{|k|} \, |k-1| + |k - |k||}{|k|(|k|-1)} \leq \frac{\left(\sqrt{|k|}+1 \right)^2}{\sqrt{|k|}(|k|-1)} . 
\end{align*}

\end{corollary}

\begin{proof}
The origin is in the exterior of the circle $ h(\partial S(1)) $ because $ \frac{1}{\sqrt{|k|}} < 1 $. The maximum value of the distance between the origin and $ h(\partial S(1)) $ is the sum of the radius and the modulus of the center. Lemma \ref{lem-center and radius of circle} implies that the radius of $ h(\partial S(1) $ is $ \frac{|k - |k||}{|k|(|k|-1)} $ and the modulus of the center is $ \frac{\sqrt{|k|} \, |k-1|}{|k|(|k|-1)} $. The triangular inequality implies that $ |k-1| \geq |k|-1 $. Then
$$ \frac{\sqrt{|k|} \, |k-1| + |k - |k||}{|k|(|k|-1)} \geq \frac{\sqrt{|k|} \, |k-1|}{|k|(|k|-1)} \geq \frac{\sqrt{|k|} }{|k|} = \frac{1}{\sqrt{|k|}} $$
Moreover,
$$ \frac{\sqrt{|k|} \, |k-1| + |k - |k||}{|k|(|k|-1)} \leq \frac{\sqrt{|k|} \,(|k| +1) + 2|k|}{|k|(|k|-1)} = \frac{\left(\sqrt{|k|}+1 \right)^2}{\sqrt{|k|}(|k|-1)} . 
$$
\end{proof}

\smallskip

\begin{rk}
Corollary \ref{cor-h(partial S1)} implies that if $ k $ is the positive real number, that is, $ g $ is the hyperbolic M\"obius map, then $ h(S(1)) $ contains the region $
\left\{ w \colon \ |w| > \frac{1}{\sqrt{|k|}} \right\} .
$
\end{rk}

\medskip

\section{Hyers-Ulam stability on the region bounded by circle}
Denote the set of natural numbers and zero, namely, $ \N \cup \{0\} $ by $ \N_0 $. Let $ F $ be the function from $ \N_0 \times \C $ to $ \C $. Let a complex valued sequence $ \{ a_n \}_{n \in \N_0 } $ satisfies the inequality
$$
| a_{n+1} - F(n,a_n) | \leq \e
$$
for a given positive number $ \e $ for all $ n \in \N_0 $ where $ | \cdot | $ is the absolute value of complex number. If there exists the sequence $ \{ b_n \}_{n \in \N_0 } $ which satisfies that 
$$ b_{n+1} = F(n,b_n) $$
for each $ n \in \N_0 $,
and $ |a_n - b_n | \leq G(\e) $ for all $ n \in \N_0 $ where the positive number $ G(\e) $ converges to zero as $ \e \rightarrow 0 $,  then we say that the sequence $ \{ b_n \}_{n \in \N_0 } $ has {\em Hyers-Ulam stability}. Denote $ F(n, z) $ by $ F_n(z) $ if necessary. 
\bigskip \\
The authors in \cite{jungnam} proved the following lemma. For the sake of completeness, we suggest the lemma and its proof.


%
%
%
%
%
%
%
%
%
%

\medskip

\begin{lemma} \label{lem-hyers ulam stability with contraction}
Let $ F \colon \N_0 \times \C \rightarrow \C $ be a function satisfying the condition
\begin{align} \label{eq-condition of F}
|F(n,u) - F(n,v)| \leq K|u-v|
\end{align}
for all $ n \in \N_0 $, $ u,v \in \C $ and for $ 0 < K < 1 $. For a given an $ \e > 0 $ suppose that the complex valued sequence $ \{a_n \}_{n \in \N_0} $ satisfies the inequality 
\begin{align}  \label{eq-sequence a-n}
|a_{n+1} - F(n,a_n)| \leq \e 
\end{align}
for all $ n \in \N_0 $. Then there exists a sequence $ \{b_n \}_{n \in \N_0} $ satisfying 
\begin{align}   \label{eq-sequence b-n}
b_{n+1} = F(n,b_n) 
\end{align}
and
\begin{align*}
|b_n - a_n| \leq K^n|b_0 - a_0| + \frac{1-K^n}{1-K} \,\e
\end{align*}
for $ n \in \N_0 $. If the whole sequence $ \{a_n \}_{n \in \N_0} $ is contained in the invariant set $ S \subset \C $ under $ F $ 
, then $ \{b_n \}_{n \in \N_0} $ is also in $ S $ under the condition, $ a_0 = b_0 $.
\end{lemma}

\begin{proof}
By induction suppose that 
$$
| b_{n-1} - a_{n-1} | \leq K^{n-1} |b_0 -a_0| + \frac{1-K^{n-1}}{1-K} \,\e .
$$
If $ n=0 $, then trivially $ |b_0 - a_0| \leq \e $. Induction implies that 
\begin{align*}
|b_n - a_n| & \leq | b_n - F(n-1,a_{n-1})| + | a_n - F(n-1,a_{n-1})| \\[0.2em]
 & \leq | F(n-1,b_{n-1}) - F(n-1,a_{n-1})| + | a_n -F(n-1,a_{n-1})| \\[0.2em]
 & = K | b_{n-1} - a_{n-1} | + \e \\
 & \leq K \left\{ K^{n-1} |b_0 -a_0| + \frac{1-K^{n-1}}{1-K} \,\e \right\} + \e \\
 &= K^{n} |b_0 -a_0| + \frac{1-K^{n}}{1-K} \,\e .
\end{align*}
Moreover, if $ a_0 = b_0 $, then the sequence $ \{b_n \}_{n \in \N_0} $ satisfies the inequality \eqref{eq-sequence a-n} without error, namely $ \e=0 $, under $ F $. Hence, $ \{b_n \}_{n \in \N_0} $ is contained in the invariant set $ S $. 
\end{proof}


%
%
%
%



\noindent The set $ S $ is called an invariant set under $ F $ (or $ S $ is invariant under $ F $) only if $ s \in S $ implies that $ F(s) \in S $. 
\medskip
\begin{lemma} \label{lem-invariant region under Mobius}
Let $ g $ be the loxodromic M\"obius map as $ g(z) = \frac{az+b}{cz+d} $ where $ a $, $ b $, $ c $ and $ d $ are complex numbers, $ ad - bc = 1 $ and $ c \neq 0 $. Let the set $ B_R $ be 
\begin{align*}
B_R &= \{ z \colon \; |z-\beta| \geq R\,|z-\alpha| \, \} \ \ \text{for} \ R>1 \\
B_R &= \{ z \colon \; |z-\beta| \geq R\,|z-\alpha| \, \} \cup \{\infty\}  \ \ \text{for} \ 0<R\leq 1 .
\end{align*}
Then $ g(B_R) \subset B_R $, that is, $ B_R $ is invariant under $ g $. 
\end{lemma}

\begin{proof}
The map $ h(z) = \frac{z-\beta}{z-\alpha} $ and $ g = h^{-1} \circ f \circ h $ where $ f(w) = kw $. Recall that $ |k|>1 $. Thus 
\begin{align*}
h(B_R) &= \{ w \colon \; |w| \geq R \,\} \cup \{ \infty \} \\[0.2em]
f \circ h(B_R) &= \{ w \colon \; |w| \geq |k|R \,\} \cup \{ \infty \} \\[0.2em]
h^{-1} \circ f \circ h(B_R) &= \{ z \colon \; |z-\beta| \geq |k|R\cdot |z-\alpha| \, \} \ \ \text{for} \ |k|R >1 \\[0.2em]
h^{-1} \circ f \circ h(B_R) &= \{ z \colon \; |z-\beta| \geq |k|R\cdot |z-\alpha| \, \} \cup \{ \infty \} \ \ \text{for} \ 0<|k|R \leq 1 .
\end{align*}
Observe that $ h^{-1} \circ f \circ h(B_R) = g(B_R) = B_{|k|R} $. Since $ R < |k|R $, we have that $ f \circ h(B_R) \subset h(B_R) $. Hence, $ g(B_R) \subset B_R $, that is, $ B_R $ is invariant under $ g $. 
\end{proof}

\begin{corollary} \label{cor-disk contained in C-BR}
If $ z \in B_R $ for $ R> 0 $, then 
$$ \left\{ z \colon \; |z-\beta| \leq \frac{R|k-1|}{|c|(R+1)\sqrt{|k|}} \right\} \subset \overline{\C \setminus B_R} . $$
\end{corollary}

\begin{proof}
The set inclusion 
\begin{align} \label{eq-set inclusion between disks}
\{ z \colon \; |z-\beta| \leq A \} \subset \{ z \colon \;|z-\alpha| \geq B \}
\end{align}
holds if and only if $ A+B \leq | \alpha - \beta | $. Let $ B $ be $ \frac{A}{R} $. Then \eqref{eq-set inclusion between disks} holds if and only if $ A \leq \frac{R}{R+1}\,| \alpha - \beta | $. Moreover, $ B \geq \frac{1}{R+1}\,| \alpha - \beta | $. Then 
$$ |z-\beta| \leq \frac{R}{R+1}\,| \alpha - \beta | \leq R |z-\alpha| .
$$
The inequality $ |z-\beta| \leq \frac{R}{R+1}\,| \alpha - \beta | $ implies that $ z \in \overline{\C \setminus B_R} $. Lemma \ref{lem-image of -d/c under h} implies that $ | \alpha - \beta | = \frac{|k-1|}{|c|\sqrt{|k|}} $. Hence, 
$$ \left\{ z \colon \; |z-\beta| \leq \frac{R|k-1|}{|c|(R+1)\sqrt{|k|}} \right\} \subset \overline{\C \setminus B_R} . $$
\end{proof}

\medskip
\begin{proposition} \label{prop-stability on D-R}
Let $ g $ be the loxodromic M\"obius map. Let $ B_R $ be the region defined in Lemma \ref{lem-invariant region under Mobius}. Let a complex valued sequence $ \{ a_n \}_{n \in \N_0 } $ satisfies the inequality
$$
| a_{n+1} - g(a_n) | \leq \e
$$
for a given $ \e > 0 $ and for all $ n \in \N_0 $. For small enough $ \e $, If $ a_0 \in B_R $ for $ R > \frac{\left(\sqrt{|k|}+1 \right)^2}{\sqrt{|k|}(|k|-1)} $, then the whole sequence $ \{ a_n \}_{n \in \N_0} $ is also contained in $ B_R $. Moreover, there exists the sequence $ \{ b_n \}_{n \in \N_0 } $ satisfying 
$$ b_{n+1} = g(b_n) $$
for each $ n \in \N $ has Hyers-Ulam stability where $ b_0 = a_0 $. 
\end{proposition}

\begin{proof}
The region $ B_R $ is invariant under $ g $ by Lemma \ref{lem-invariant region under Mobius}. Define the distance between the circle, $ \partial B_{r_1} $ and the disk (or the exterior of the disk), $ B_{r_2} $ as follows
$$
\mathrm{dist} (B_{r_1},\,\partial B_{r_2}) = \inf \{ \,|z-w| \, \colon \; z \in B_{r_1} \ \text{and} \ w \in \partial B_{r_2} \} .
$$
Recall that $ h(B_R) \subset f\circ h(B_R) $. Thus $ h (B_{|k|R}) $ and $ \partial h(B_{R}) $ are disjoint and $ \mathrm{dist} (h (B_{|k|R}),\, \partial h(B_{R})) = (|k|-1)R > 0 $. Then since $ h $ is a conformal isomorphism between $ B_R $ and $ h(B_R) $, the distance $ \mathrm{dist} \left(B_{|k|R},\, \partial B_{R} \right) $ is also positive number. Choose $ \e>0 $ satisfying $ \e < \mathrm{dist} \left(B_{|k|R},\, \partial B_{R} \right) $. 
Suppose that $ a_0 $ is in $ B_R $. Then $ g(a_0) \in B_{|k|R} $ and $ \{ z \colon \; |z-g(a_0)| \leq \e \} $ is a subset of $ B_{R} $, that is, $ a_1 \in B_R $. By induction, the whole sequence $ \{ a_n \}_{n \in \N_0 } $ is contained in $ B_R $.
\smallskip \\
Recall that $ g'(z) = \frac{1}{(cz+d)^2} $. The definition of $ S(r) $ in \eqref{eq-definition of Sr} implies that $ |g'(z)| < 1 $ if and only if $ z \in S(1) $. Corollary \ref{cor-h(partial S1)}, the definition of $ B_R $ in Lemma \ref{lem-invariant region under Mobius} and the condition $ R > \frac{\left(\sqrt{|k|}+1 \right)^2}{\sqrt{|k|}(|k|-1)} $ implies that $ h(S(1)) \supset h(B_R) $. Then $ |g'(z)| < 1 $ for all $ z \in B_R $. 
Lemma \ref{lem-hyers ulam stability with contraction} implies that 
\begin{align*}
|b_n - a_n| \leq K^n|b_0 - a_0| + \frac{1 + K^n}{1 - K} \,\e
\end{align*}
where $ \displaystyle K = \inf_{z \in B_R} \{ |g'(z)| \} < 1 $. Hence, the sequence $ \{ b_n \}_{n \in \N_0} $ has Hyers-Ulam stability where $ b_0 = a_0 $. 
\end{proof}

\smallskip

\section{Avoided region} \label{sec-Avoided region}
The map $ g $ is the M\"obius map with complex coefficients  
$$ g(z) = \frac{az + b}{cz + d} $$
for $ ad - bc =1 $ and $ c \neq 0 $. Since the point $ \infty $ is not a fixed point of $ g $, the preimage of $ \infty $ under $ g $, namely, $ g^{-1}(\infty) $ is in the complex plane. For a given $ \e > 0 $, let $ \{ a_n \}_{n \in \N_0} $ be the sequence which satisfies the inequality 
$$ | a_{n+1} - g(a_n)| \leq \e $$ 
for all $ n \in \N_0 $. If $ \{ a_n \}_{n \in \N_0} $ contains $ g^{-1}(\infty) $, say $ a_k $, then $ | a_{k+1} - \infty | $ is not bounded where $ | \cdot | $ is the absolute value of the complex number. 
In order to exclude $ g^{-1}(\infty) $ in the sequence $ \{ a_n \}_{n \in \N_0 } $, the region $ \RR_g $ is considered such that if the initial point of the sequence, $ a_0 $ is not in $ \RR_g $, then the whole sequence $ \{ a_n \}_{n \in \N_0 } $ cannot be in the same region $ \RR_g $. Let the forward orbit of $ p $ under $ F $ be the set $ \{ F(p),F^2(p), \ldots ,F^n(p) , \ldots \} $ and denote it by $ \Orb_{\N}(p, F) $. 
\bigskip 
\begin{definition} \label{def-avoided region}
Let $ F $ be the map on $ \hat{\C} $ which does not fix $ \infty $. Let $ \{ a_n \}_{n \in \N_0 } $ be any sequence which satisfies $ | a_{n+1} - F(a_n)| \leq \e $ for a given $ \e >0 $. Avoided region $ \RR_F \subset \C $ is defined as follows
\begin{enumerate}
\item $ \hat{\C} \setminus \RR_F $ is (forward) invariant under $ F $, that is, $ F(\hat{\C} \setminus \RR_F) \subset \hat{\C} \setminus \RR_F $.
\item For any given initial point $ a_0 $ in $ \C \setminus \RR_F $, all points in the sequence $ \{ a_n \}_{n \in \N_0 } $ satisfying $ | a_{n+1} - F(a_n)| \leq \e $ are in $ \C \setminus \RR_F $. 
\end{enumerate}
If $ \RR_F $ contains $ \Orb_{\N}(p, F^{-1}) $ where $ p \in \hat{\C} $, then it is called the avoided region at $ p $ and is denoted by $ \RR_F(p) $. 
\end{definition}
\smallskip 
\noindent In the above definition, the avoided region does not have to be connected.  
\begin{remark}
The set $ \hat{\C} \setminus B_R $ in Proposition \ref{prop-stability on D-R} for $ R > \frac{\left(\sqrt{|k|}+1 \right)^2}{\sqrt{|k|}(|k|-1)} $ is an avoided region at $ \infty $. However, avoided region $ \hat{\C} \setminus B_R $ can be extended to some neighborhood of $ \Orb_{\N}(\infty, g^{-1}) $, which is denoted to be $ \RR_{g}(\infty) $. 
\end{remark}
%
%
%
%
%
\medskip
\begin{lemma} \label{lem-1st lemma for avoided region}
Let $ f $ be the map $ f(w) = kw $ for $ |k| > 1 $. Let $ \{ c_n \}_{n \in \N_0} $ be the sequence for a given $ \delta > 0 $ satisfying that
\begin{align} \label{eq-sequence c-n for f}
| c_{n+1} - f(c_n) | \leq \delta_0 
\end{align}
for all $ n \in \N_0 $. If $ \left| c_j - \frac{1}{k^m} \right| > \frac{t\delta_0}{|k| - 1} $, then $ \left| c_{j+1} - \frac{1}{k^{m-1}} \right| > \frac{t\delta_0}{|k| - 1} $ for each $ j,m \in \N_0 $ and for all $ t \geq 1 $. 
\end{lemma}

\begin{proof}
The inequality $ | c_{j+1} - f(c_j) | \leq \delta_0 $ implies that 
\begin{align*}
\delta_0 &\geq | c_{j+1} - f(c_j) | \\
&= \left| \left( c_{j+1} - \frac{1}{k^{m-1}} \right) - \left( kc_j - \frac{1}{k^{m-1}} \right) \right| \\
&= \left| \left( c_{j+1} - \frac{1}{k^{m-1}} \right) - k\left( c_j - \frac{1}{k^{m}} \right) \right| \\
&\geq \left| \,\left| c_{j+1} - \frac{1}{k^{m-1}} \right| - |k| \left| c_j - \frac{1}{k^{m}} \right|\, \right|
\end{align*} 
Then 
\begin{align*} 
\left| c_{j+1} - \frac{1}{k^{m-1}} \right| \geq |k| \left| c_j - \frac{1}{k^{m}} \right| - \delta_0 > |k|\dfrac{t\delta_0}{|k|-1} - \delta_0 \geq \dfrac{|k|t\delta_0}{|k|-1} - t\delta_0 = \dfrac{t\delta_0}{|k|-1} 
\end{align*}
for $ t \geq 1 $. 
\end{proof}
\noindent Let the region
\begin{align} \label{eq-disk around the backward orbit of infty}
\DD_n(\delta) = \left\{ w \colon \,\left| w - \frac{1}{k^n} \right| < \delta \right\}
\end{align}
for $ n \in \N_0 $ and $ \delta > 0 $. 
%
%
\begin{figure}
    \centering
    \includegraphics[scale=0.8]{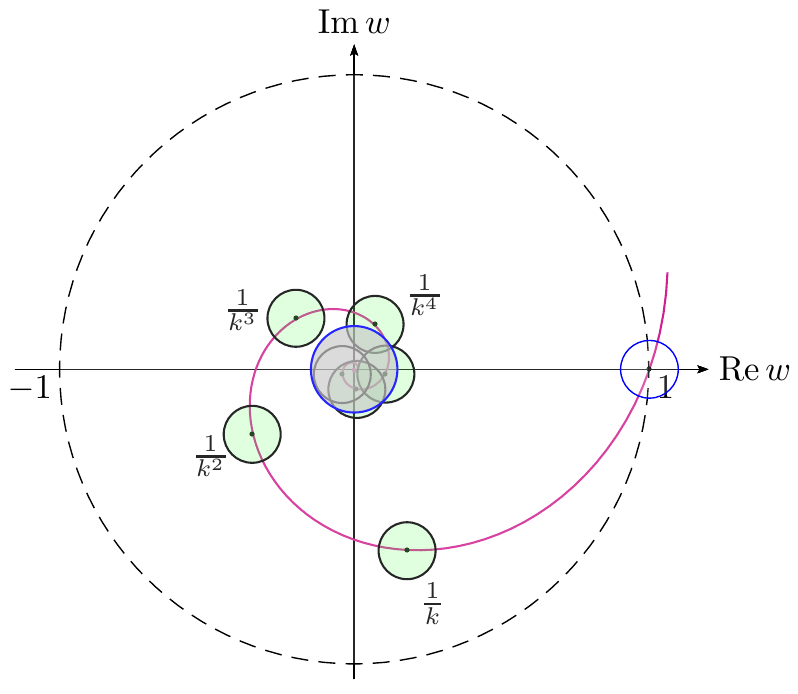}   
    \caption{Avoided region $ \RR_f(1) $}
    \label{fig:avoided region Rf1}
\end{figure}
%
%
\begin{lemma} \label{lem-2nd lemma for avoided region}
Let $ f $ be the map $ f(w) = kw $ for $ |k| > 1 $. Let $ \{ c_n \}_{n \in \N_0} $ be the sequence for a given $ \delta > 0 $ satisfying that
\begin{align*} 
| c_{n+1} - f(c_n) | \leq \delta_0 
\end{align*}
for all $ n \in \N_0 $. Let $ \delta $ be a positive number $ \frac{t\delta_0}{|k|-1} $ for some $ t>1 $. Define the set 
$$ \DD = \bigcup_{n=1}^{\infty} \DD_n(\delta)
$$
where $ \DD_n $ is the closed disk defined in \eqref{eq-disk around the backward orbit of infty}. If $ c_0 $ is contained in $ \C \setminus \DD $, then the sequence $ \{ c_n \}_{n \in \N_0} $ is contained in $ \C \setminus \DD $, that is, $ \DD $ an avoided region of $ g $ at one. 
\end{lemma}

\begin{proof}
The definition of $ \DD_n(\delta) $ can be extended to the case that $ n $ is integer. Thus $ f(\DD_n(\delta)) = \DD_{n-1}(|k|\delta) $ for $ n \in \Z $. Since $ |k| >1 $, $ f(\DD) $ is as follows
\begin{align*}
f(\DD) = f\left(\bigcup_{n=1}^{\infty} \DD_n(\delta) \right) = \bigcup_{n=1}^{\infty} f(\DD_n(\delta)) = \bigcup_{n=0}^{\infty} \DD_n(|k|\delta) \supset \DD 
\end{align*}
Then $ \C \setminus \DD $ is invariant under $ f $. Assume that $ c_0 \in \C \setminus \DD $. Then we have 
\begin{align*}
c_0 \in \left\{ w \colon \left| w - \frac{1}{k^m} \right| \geq \delta > \frac{t\delta_0}{|k|-1} \right\} 
\end{align*}
for all $ m \in \N $. Lemma \ref{lem-1st lemma for avoided region} and the induction implies that for each $ N \in \N_0 $ 
\begin{align*}
c_N \in \left\{ w \colon \left| w - \frac{1}{k^{m-N}} \right| \geq \delta \right\} 
\end{align*}
for all $ m \in \N $. Since 
$$ \C \setminus \DD \subset \bigcup_{m=1}^{\infty} \left\{ w \colon \left| w - \frac{1}{k^{m-N}} \right| \geq \delta \right\} $$ 
for every $ N \in \N_0 $, the set $ \DD $ is an avoided region of $ g $ at one. 
\end{proof}

\medskip 

\noindent The origin is the accumulation point of the set $ \left\{ \frac{1}{k^m} \right\} $ for $ m \in \N_0 $ and is the repelling fixed point of $ f $. Then we may choose the avoided region of $ f $ at one as follows
\begin{align} \label{eq-avoided region of f}
\RR_f(1) = \big(\C \setminus D \left(|k|\delta \right)\big) \cup \bigcup_{n=1}^N \DD_n(\delta)
\end{align}
where $ D \left(|k|\delta \right) = \{ w \colon \, |w| \geq |k|\delta \} $ and $ N > \frac{\log \left( |k|-1 \right) + \log \delta}{\log|k|} $. 

\medskip

\begin{proposition} \label{prop-avoided region for g}
Let $ \RR_f(1) $ the avoided region defined in \eqref{eq-avoided region of f}. For a given $ \e > 0 $, let $ \{ a_n \}_{n \in \N_0} $ be the sequence which satisfies the inequality 
$$ | a_{n+1} - g(a_n)| \leq \e $$ 
for all $ n \in \N_0 $ where $ g $ is the loxodromic M\"obius map. Then for sufficiently small $ \e > 0 $, an avoided region of $ g $ at $ \infty $ is $ h^{-1}(\RR_f(1)) $. The number $ \e $ depends only on $ |k| $, $ \delta $ and $ |c| $. 
\end{proposition}
\noindent The proof of the above proposition requires the following lemma.
\bigskip
\begin{lemma} \label{lem-3rd lemma for avoided region}
Let $ h $ be the map $ h(z) = \frac{z-\beta}{z-\alpha} $. Let the disk of which center is $ p $ and radius is $ r $ be as follows 
\begin{align*}
U(p,r) = \{ z \colon \ |z-p| < r \}
\end{align*}
Assume that $ |\beta - p| - r \neq 0 $. Then the radius of $ h(U(p,r)) $ is $ \frac{r^2|\alpha - \beta|^2}{\big| |\beta - p|^2 - r^2 | \big|} $. 
\end{lemma}

\begin{proof}
The equation $ w = h(z) $ holds if and only if $ z = \frac{\beta w - \alpha}{w-1} $. Thus the following equivalent inequalities hold
\begin{align*}
|z-p|<r & \Longleftrightarrow \left| \frac{\beta w - \alpha}{w-1} - p \right| < r \\[0.2em]
& \Longleftrightarrow \left| \beta w - \alpha - p(w-1) \right| < r|w-1| \\[0.2em]
& \Longleftrightarrow \left| (\beta - p) w - \alpha + p \right|^2 < r^2|w-1|^2 \\[0.2em]
& \Longleftrightarrow |\beta - p|^2|w|^2 + \left( \overline{-\alpha + p} \right)(\beta - p) w + \left( -\alpha + p \right)\left( \overline{\beta - p} \right) \overline{w} \\
& \qquad \qquad + |-\alpha + p|^2 < r^2 \left(|w|^2 -w - \overline{w} + 1\right) \\[0.2em]
& \Longleftrightarrow \left( |\beta - p|^2 -r^2 \right)|w|^2 + \left\{ \left( \overline{-\alpha + p} \right)(\beta - p) + r^2 \right\} w \\
& \qquad \qquad + \left\{ ( -\alpha + p)( \overline{\beta - p}) + r^2 \right\} \overline{w} < r^2 - | -\alpha +p|^2
\end{align*}
Suppose firstly that $ |\beta - p| -r > 0 $. Then 
\begin{align*}
&\left| \sqrt{|\beta-p|^2 - r^2}\, w - \frac{(-\alpha + p)( \overline{\beta - p}) + r^2}{\sqrt{|\beta-p| - r^2}} \right| \\
& \qquad < \frac{|(-\alpha + p)( \overline{\beta - p}) + r^2|^2}{|\beta-p|^2 - r^2} + r^2 - | -\alpha +p|^2 \\
& \qquad = \frac{|(-\alpha + p)( \overline{\beta - p}) + r^2|^2 + \left( r^2 - | -\alpha +p|^2 \right) ( |\beta-p|^2 - r^2 )}{|\beta-p|^2 - r^2}
\end{align*}
The numerator of the radius is that
\begin{align*}
& \quad \ |-\alpha + p|^2|\beta -p|^2 + r^2 (\overline{-\alpha + p})(\beta - p) + r^2(-\alpha + p)(\overline{\beta - p}) + r^4 \\
& \qquad + r^2 |\beta-p|^2 - |-\alpha + p|^2|\beta -p|^2 - r^4 + r^2|-\alpha +p|^2 \\[0.2em]
& = r^2 \big\{ (\overline{-\alpha + p})(\beta - p) + (-\alpha + p)(\overline{\beta - p}) + |\beta-p|^2 + |-\alpha +p|^2 \big\} \\[0.2em]
& = r^2 \big( -\overline{\alpha}\beta + \overline{p}\beta + \overline{\alpha}p - |p|^2 - \alpha\overline{\beta} + p\overline{\beta} + \alpha\overline{p} - |p|^2  \\
& \qquad + |\beta|^2 - p\overline{\beta} - \overline{p}\beta + |p|^2 + |\alpha|^2 - p\overline{\alpha} - \overline{p}\alpha + |p|^2 \big) \\[0.2em]
& = r^2 \big( -\overline{\alpha}\beta - \alpha\overline{\beta} +  |\beta|^2 + |\alpha|^2 \big) \\[0.2em]
& = r^2 | \alpha - \beta |^2
\end{align*}
Then the radius is $ \frac{r^2 | \alpha - \beta |^2}{|\beta-p|^2 - r^2} $ only if $ |\beta-p| - r > 0 $. Similarly, if $ |\beta-p| - r > 0 $, then the radius is $ \frac{r^2 | \alpha - \beta |^2}{r^2 - |\beta-p|^2} $. Hence, the radius of $ h(U(p,r)) $ is $$ \frac{r^2 | \alpha - \beta |^2}{\big| |\beta-p|^2 - r^2 \big|} . $$
\end{proof}

%
\begin{proof}[proof of Proposition \ref{prop-avoided region for g}]
The invariance of $ \C \setminus h^{-1}(\RR_f(1)) $ under $ g $ is equivalent to the fact that $ g(h^{-1}(\RR_f(1))) \supset h^{-1}(\RR_f(1)) $. Recall that $ f \circ h = h \circ g $. Then since $ \RR_f(1) $ is an avoided region of $ f $, $ f(\RR_f(1)) \supset \RR_f(1) $. Then
$$ g(h^{-1}(\RR_f(1))) = g \circ h^{-1}(\RR_f(1)) = h^{-1} \circ f(\RR_f(1)) \supset h^{-1}(\RR_f(1)) . $$
By the above set inclusion, $ \C \setminus h^{-1}(\RR_f(1)) $ is invariant under $ g $. Recall that $ h^{-1}(\C \setminus D(|k|\delta)) = B_{|k|\delta} $. We may choose an $ \e>0 $ which satisfies that $ \e < \mathrm{dist} \left(\partial B_{|k|\delta},B_{|k|^2\delta} \right) $. Thus if $ a_0 \in B_{|k|\delta} $, the sequence $ \{a_n\}_{n \in \N_0} $ is also contained in $ B_{|k|\delta} $ by induction. Choose $ a_0 \in \C \setminus h^{-1}(\RR_f(1)) $ and suppose that $ a_n $ is contained in the same set. Since $ |a_{n+1} -g(a_n)| \leq \e $, it suffice to show that 
\begin{align*}
\{ z \colon \;|z - g(a_n)| \leq \e \} \cap h^{-1}(\RR_f(1)) = \varnothing 
\end{align*}
for small enough $ \e > 0 $. Observe that $ a_n \in \C \setminus h^{-1}(\RR_f(1)) $ if and only if $ h(a_n) \in\C \setminus \RR_f(1) $. Lemma \ref{lem-2nd lemma for avoided region} implies that 
\begin{align} \label{eq-disjoint disk and avoided region}
\{ w \colon \;|w - f(h(a_n))| \leq \delta_0 \} \cap \RR_f(1) = \varnothing .
\end{align}
Denote the set $ \{ z \colon \;|z - g(a_n)| \leq \e \} $ by $ U_n $ and $ \{ w \colon \;|w - f(h(a_n))| \leq \delta_0 \} $ by $ V_n $. Find the small $ \e>0 $ which implies that $ h(U_n) \subset V_n $. Lemma \ref{lem-3rd lemma for avoided region} implies that the radius of $ g(U_n) $ is $ \frac{\e^2 | \alpha - \beta |^2}{|\beta - a_n|^2 - \e^2} $. Lemma \ref{lem-image of -d/c under h} implies that $ |\alpha - \beta|^2 = \frac{|k-1|^2}{|c|\sqrt{|k|}} $ and Corollary \ref{cor-disk contained in C-BR} implies that $ |\beta - a_n| \geq \frac{\delta\sqrt{|k|}|k-1|}{|c|(\delta|k|+1)} $. Thus an upper bound of the radius of $ g(U_n) $ is as follows 
\begin{align*}
\frac{\e^2 | \alpha - \beta |^2}{|\beta - a_n|^2 - \e^2} 
\leq  \frac{\e^2}{\left( \frac{\delta\sqrt{|k|}|k-1|}{|c|(\delta|k|+1)} \right)^2-\e^2} \cdot \frac{|k-1|^2}{|c|\sqrt{|k|}} 
\end{align*}
where $ \e < \frac{\delta\sqrt{|k|}|k-1|}{|c|(\delta|k|+1)} $. Moreover, since $ h \circ g = f \circ h $, the center of $ V_n $, namely, $ f \circ h(a_n) = h \circ g(a_n) $ is contained in the interior of $ h(U_n) $. Thus $ h(U_n) \subset V_n $ if and only if 
$$ \mathrm{diameter \ of} \ g(U_n) \leq \mathrm{radius \ of} \ V_n . 
$$
Then choose $ \e > 0 $ satisfying the inequality
$ \frac{\e^2}{\left( \frac{\delta\sqrt{|k|}|k-1|}{|c|(\delta|k|+1)} \right)^2-\e^2} \cdot \frac{|k-1|^2}{|c|\sqrt{|k|}} < \delta 
$ or equivalently, choose $ \e < \frac{\delta \sqrt{|k|}|k-1|^2}{\{2|k-1|^2 + 2|c|\delta \sqrt{|k|}\}^{\frac{1}{2}} |c|(\delta|k|+1)} $. Then take $ \e > 0 $ as follows 
\begin{align} \label{eq-upper bound of epsilon}
\e < \min \left\{ \mathrm{dist} \left(\partial B_{|k|\delta},B_{|k|^2\delta} \right),\ \frac{\delta \sqrt{|k|}|k-1|^2}{\left\{ 2|k-1|^2 + 2|c|\delta \sqrt{|k|} \right\}^{\frac{1}{2}} |c|(\delta|k|+1)} \right\} .
\end{align}
The equation \eqref{eq-disjoint disk and avoided region} implies that $ (V_n) \cap (\RR_f(1)) = \varnothing $. Moreover, if we choose $ \e >0 $ in \eqref{eq-upper bound of epsilon}, then $ h(U_n) \subset V_n $. 
\begin{align*}
U_n \cap h^{-1}(\RR_f(1)) &\subset h^{-1}(V_n) \cap h^{-1}(\RR_f(1)) \\
&= h^{-1}(V_n \cap \RR_f(1)) \\
&= \varnothing
\end{align*}
Hence, by induction $ \{a_n\}_{n \in \N_0} $ is contained in $ \C \setminus \RR_f(1) $, which is an avoided region of $ g $ at $ \infty $. 
\end{proof}

\begin{remark}
The forward orbit of $ 1 $ under $ f^{-1} $ is $ \left\{ \frac{1}{k^n} \colon n \in \N \right\} $. The set $ \left\{ h^{-1}\left(\frac{1}{k^n}\right) \colon n \in \N \right\} $ is the forward orbit of $ \infty $ under $ g^{-1} $ because $ h^{-1}(1) = \infty $ and $ g^{-1} = h^{-1} \circ f^{-1} \circ h $. Then the avoided region $ \RR_g(\infty) $ can be chosen as $ h^{-1}(\RR_f(1)) $. 
\end{remark}

\medskip

\section{Escaping time from the region}

Let the sequence $ \{ c_n \}_{n \in \N_0 } $ satisfies the following
\begin{align} \label{eq-sequence under f}
| c_{n+1} - f(c_n) | \leq \delta_0
\end{align}
for all $ n \in \N_0 $. For the given set $ S $, assume that $ c_0 \in S $. If the distance between $ c_n $ and the closure of $ S $ is positive for all $ n \geq N $, then $ N $ is called {\em escaping time} of the sequence $ \{ c_n \}_{n \in \N_0 } $ from $ S $ under $ f $. If the escaping time $ N $ is independent of the initial point $ c_0 $ in $ S $, then $ N $ is called {\em uniformly} escaping time. 
\begin{lemma} \label{lem-escaping time under f}
Let $ \{ c_n \}_{n \in \N_0 } $ be the sequence defined in the equation \eqref{eq-sequence under f} where $ f(w) = kw $ for $ |k| > 1 $. Denote the set $ \big( \C \setminus D(R) \big) \setminus \RR_f(1) $ by $ E_f $ where $ R > \frac{\left(\sqrt{|k|} +1 \right)^2}{\sqrt{|k|}(|k|-1)} $. Then the (uniformly) escaping time $ N $ from the region $ E_f $ under $ f $ satisfies the following inequality
\begin{align*}
N > \log \left\{ \frac{1}{\delta_0} \left(\frac{2\left(\sqrt{|k|} +1 \right)^2}{\sqrt{|k|}(|k|-1)} + 1 \right)  + 1 \right\} \bigg/ \log |k| 
\end{align*}
for small enough $ \delta_0 > 0 $. 
\end{lemma}

\begin{proof}
By triangular inequality, we have 
\begin{align*}
| f^n(c_0) - c_0 | &\leq  | f^n(c_0) - f^{n-1}(c_1) | + | f^{n-1}(c_1) - f^{n-2}(c_2) | + \\
 & \qquad \cdots + | f^2(c_{n-2}) - f(c_{n-1})| +  | f(c_{n-1}) - c_{n}| + |c_n - c_0 | \\
&= \sum^{n}_{j=1} |k|^{n-j}|f(c_{j-1}) - c_j| + |c_n - c_0 | \\
&\leq \frac{|k|^n-1}{|k|-1}\, \delta_0 + |c_n - c_0 | .
\end{align*}
Recall that we may choose $ \delta $ which satisfies that $ \delta \geq \frac{\delta_0}{|k|-1} $ by Lemma \ref{lem-2nd lemma for avoided region} and $ D(\delta|k|) \subset \RR_f(1) $. Since $ c_0 \in E_f $, we have $ |c_0| \geq \frac{|k|\delta_0}{|k|-1} $. Then we obtain that  
\begin{align*}
|c_n - c_0 | &\geq | f^n(c_0) - c_0 | -\frac{|k|^n-1}{|k|-1}\, \delta_0 \\
&= (|k|^n-1)|c_0| - \frac{|k|^n-1}{|k|-1}\,\delta_0 \\
&= (|k|^n-1) \left(|c_0| -\frac{\delta_0}{|k|-1} \right) \\
&\geq (|k|^n-1) \left(\frac{|k|\delta_0}{|k|-1} -\frac{\delta_0}{|k|-1} \right) \\
&= (|k|^n-1)\, \delta_0
\end{align*}
If $ |c_n - c_0 | $ is greater than the diameter of $ \big( \C \setminus D(R) \big) \setminus \RR_f(1) $ for all $ n \geq N $, then the escaping time is $ N $. Observe that $ \C \setminus D(R) $ is the disk with radius $ R $. Thus the diameter of $ \C \setminus D(R) $ is twice of the radius. Then the inequality $ (|k|^N-1)\, \delta_0 > \frac{2\left(\sqrt{|k|} +1 \right)^2}{\sqrt{|k|}(|k|-1)} $ hold for the escaping time $ N $. Hence, $ N $ is the uniformly escaping time satisfying $ N > \frac{\log \Big\{ \frac{1}{\delta_0} \Big(\frac{2(\sqrt{|k|} +1)^2}{\sqrt{|k|}(|k|-1)} + 1 \Big)  + 1 \Big\}}{\log |k|} $. 
\end{proof}

\medskip

\begin{remark}
Observe that the inequality $ \frac{2 \left(\sqrt{|k|} +1 \right)^2}{\sqrt{|k|}(|k|-1)} + 1 \leq \left( \frac{\sqrt{|k|}+1}{\sqrt{|k|}-1} \right)^3 $ for $ |k| > 1 $. This inequality suggest a sufficient condition of the uniformly escaping time 
$$ N > \log \left\{ \frac{1}{\delta_0} \left( \frac{\sqrt{|k|}+1}{\sqrt{|k|}-1} \right)^3 +1 \right\} \bigg/ \log |k| . $$
\end{remark}

\medskip

\noindent The sequence $ \{ a_n \}_{n \in \N_0} $ is defined as the set each of which element $ a_n = h(c_n) $ for every $ n \in \N_0 $ where $ \{ c_n \}_{n \in \N_0} $ is defined in \eqref{eq-sequence under f}. Recall that $ f $ is the map $ h \circ g \circ h^{-1} $. Denote the radius of the ball $ h^{-1} \big( B(c_j, \delta) \big) $ by $ \e_j $ for $ j \in \N_0 $. Then the sequence $ \{ a_n \}_{n \in \N_0} $ as follows  
\begin{align}  \label{eq-sequence a-i without epsilon}
| a_{n+1} - g(a_n) | \leq \e_n .
\end{align}
corresponds $ \{ c_n \}_{n \in \N_0} $ by the conjugation $ h $. Then the escaping time of $ \{ a_n \}_{n \in \N_0} $ from $ h^{-1}(E_f) $ under $ g $ is the same as that of $ \{ c_n \}_{n \in \N_0} $ from $ E_f $ under $ f $ in Lemma \ref{lem-escaping time under f}. Furthermore, since $ h $ is uniformly continuous on the closure of the set, $ h^{-1}\big( \big(\C \setminus D(R)\big) \setminus \RR_f(1)\big) $ for $ R > \frac{\left(\sqrt{|k|} +1 \right)^2}{\sqrt{|k|}(|k|-1)} $ under Euclidean metric, there exists $ \e > 0 $ such that $ h \big( B(a_j,\e) \big) \subset B(c_j,\delta) $ for $ j =1,2, \ldots N_1 $ for all $ c_j \in E_f $. 
Thus we obtain the following Proposition.
\bigskip
%
\begin{figure}
    \centering
    \begin{subfigure}[b]{0.45\textwidth}
        \includegraphics[width=\textwidth]{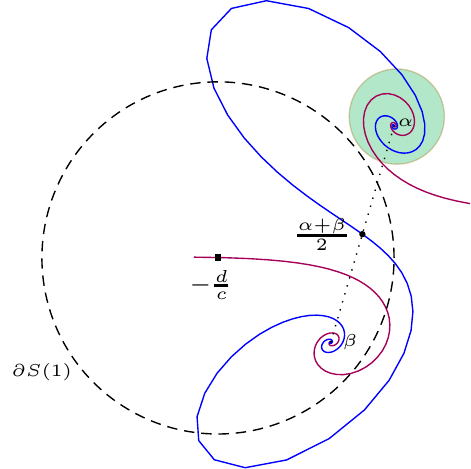}
        \caption{Invariant spirals under $ g $}
        \label{fig:invariant spirals under g}
    \end{subfigure} \
    ~ 
    \begin{subfigure}[b]{0.45\textwidth}
        \includegraphics[width=\textwidth]{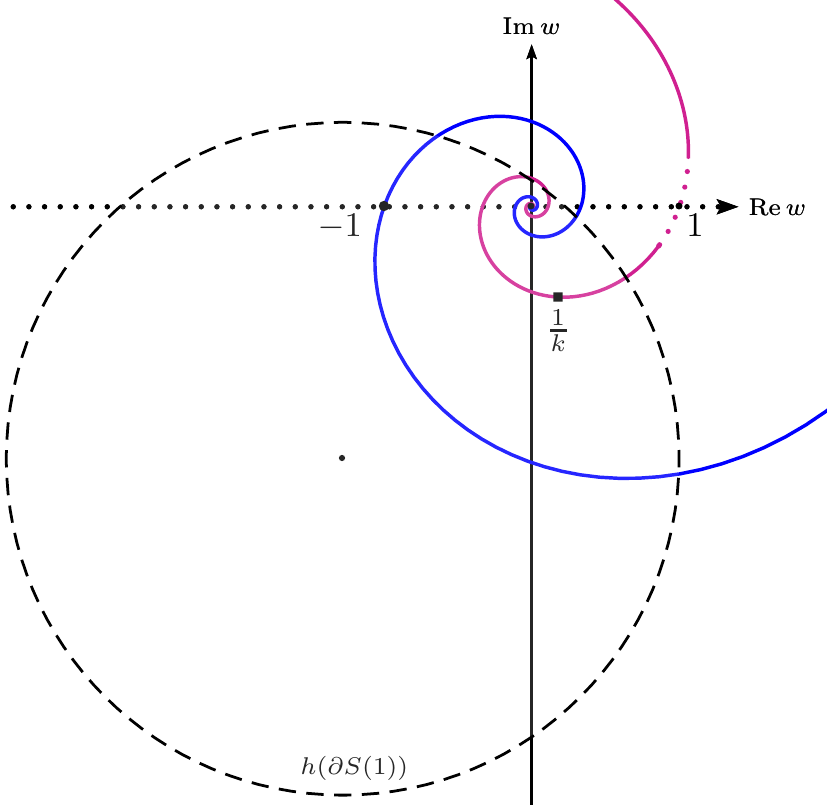}
        \caption{Invaiant spirals under $ f $}
        \label{fig:invaint sprials under f}
    \end{subfigure}
    \caption{Invariant spirals under $ g $ and its image under $ h $}\label{fig:invariant spirals under g and its images}
\end{figure}

%
\begin{proposition} \label{prop-same escaping time under g}
Let $ \{ c_n \}_{n \in \N_0} $ be the sequence satisfying 
\begin{align*}  
| c_{n+1} - f(c_n) | \leq \delta_0 
\end{align*}
where $ f(w) = kw $ for $ |k|>1 $ on $ E_f $ defined in Lemma \ref{lem-escaping time under f}. Let $ N $ be the (uniformly) escaping time from $ E_f $ under $ f $. Let $ \{ a_n \}_{n \in \N_0} $ be the sequence satisfying $ a_n = h(c_n) $ for every $ n \in \N_0 $. Then there exists $ \e > 0 $ such that if $ \{ a_n \}_{n \in \N_0} $ satisfies that 
\begin{align}  \label{eq-sequence a-i without epsilon 2}
| a_{n+1} - g(a_n) | \leq \e 
\end{align}
on $ h^{-1}(E_f) $ for $ n = 1,2,\ldots, N-1 $, then the escaping time from $ h^{-1}(E_f) $ under $ g $ is also $ N $.
\end{proposition}


\begin{remark}
The set $ h^{-1}(E_f) $ is as follows
$$ h^{-1}\big(\big(\C \setminus D(R)\big) \setminus \RR_f(1)\big) = \big(\C \setminus B_R \big) \setminus h^{-1}(\RR_f(1)) =  \big(\C \setminus B_R \big) \setminus \RR_g(\infty) . $$
\end{remark}

\medskip

\section{Hyers-Ulam stability on the complement of the  avoided region}
Hyers-Ulam stability of loxodromic M\"obius map on the region $ B_R $ where $ R > \frac{(\sqrt{|k|}+1)^2}{\sqrt{|k|}(|k|-1)} $ is proved in Proposition \ref{prop-stability on D-R}. 
In this section, we show Hyers-Ulam stability on the region $ \big(\C \setminus B_R\big) \setminus \RR_g(\infty) $ where $ \RR_f(\infty) $ is the avoided region defined in Section \ref{sec-Avoided region} for the finite time bounded above by the uniformly escaping time. Then combining the stability on complement region and Proposition \ref{prop-stability on D-R} implies Hyers-Ulam stability of $ g $ on the set $ \hat{\C} \setminus \RR_{g}(\infty) $. 
%
%
%
\medskip

\begin{lemma} \label{lem-image of disk 1/k under h-1}
Let $ \DD_1 $ be the disk $ \left\{ w \colon \left| w- \frac{1}{k} \right| < \delta \right\} $. Then the disk $ h^{-1}(\DD_1) $ is $ \left\{ z \colon \; \left| z+\frac{d}{c} \right| < \delta \left| \frac{k}{k-1} \right||z - \alpha| \right\} $. Moreover, the center of the disk $ h^{-1}(\DD_1) $ is 
\begin{align*}
\frac{|k-1|^2}{|k-1|^2-\delta^2|k|^2} \left( -\frac{d}{c} \right) + \frac{-\delta^2|k|^2 \alpha}{|k-1|^2-\delta^2|k|^2}\,
\end{align*}
and the radius is 
\begin{align*}
\frac{\delta|k|\sqrt{|k|}|k-1|}{|c|(|k-1|^2 + \delta^2|k|^2)}
\end{align*}
for small enough $ \delta > 0 $. 
\end{lemma}
\begin{proof}
The map $ h(z) = \frac{z-\beta}{z-\alpha} $. Recall that $ (c\alpha+d)(c\beta+d) = 1 $ by equation \eqref{eq-ad-bc is 1} in  Lemma \ref{lem-fixed points of loxo Mobius map} and $ k = \frac{1}{(c\beta+d)^2} $ by Lemma \ref{lem-congugation of loxo Mobius map}. Lemma \ref{lem-image of -d/c under h} implies that $ [\;\!c(\alpha - \beta)]^2 = \frac{(k-1)^2}{k} $. Thus
\begin{align*}
w - \frac{1}{k} &= \frac{z-\beta}{z-\alpha} - \frac{c\beta+d}{c\alpha+d} \\[0.2em]
&= \frac{(z-\beta)(c\alpha+d) - (z-\alpha)(c\beta+d)}{(z-\alpha)(c\alpha+d)} \\
&= \frac{(cz+d)(\alpha-\beta)}{(z-\alpha)(c\alpha+d)} \\
&= \frac{z-\left(-\frac{d}{c}\right)}{z-\alpha} \cdot \frac{c(\alpha-\beta)}{c\alpha+d} .
\end{align*}
Then the following equation holds
\begin{align*}
\left| w - \frac{1}{k} \right| = \left| \frac{z-\left(-\frac{d}{c}\right)}{z-\alpha} \right| \frac{|k-1|}{|k|} .
\end{align*}
Hecne, $ \left| w - \frac{1}{k} \right| < \delta $ implies that $$ h^{-1}(\DD_1) = \left\{ z \colon \; \left| z+\frac{d}{c} \right| < \delta \left| \frac{k}{k-1} \right||z - \alpha| \right\} . $$ %
The straight line 
$ \ell (t) = \left\{ t\alpha + (1-t)\left(-\frac{d}{c}\right) \colon t \in \R \right\} $ goes through the center of $ h^{-1}(\DD_1) $. Thus the intersection points in $ \ell (t) $ and $ h^{-1}(\partial \DD_1) $ is the endpoints of the diameter of $ h^{-1}(\DD_1) $. Solve the following equation
\begin{align*}
 \left| t\alpha + (1-t)\left(-\frac{d}{c}\right) + \frac{d}{c} \right| &= \delta \left| \frac{k}{k-1} \right| \left|t\alpha + (1-t)\left(-\frac{d}{c}\right) - \alpha \right| \\[0.2em]
|t| \left| \alpha + \frac{d}{c} \right| &= \delta \left| \frac{k}{k-1} \right| |1-t| \left| \alpha + \frac{d}{c} \right| .
\end{align*}
Thus $ t = 1 - \frac{|k-1|}{|k-1| \pm \delta |k|} $. Then $ \ell(t) \cap h^{-1}(\DD_1) $ the set of the following two points
\begin{align}  \label{eq-end points of diameter}
& \frac{|k-1|}{|k-1| + \delta|k|} \left(-\frac{d}{c}\right) + \frac{\delta|k|}{|k-1| + \delta|k|} \alpha \\[0.2em]
& \qquad \qquad \text{and} \ \ \frac{|k-1|}{|k-1| - \delta|k|} \left(-\frac{d}{c}\right) - \frac{\delta|k|}{|k-1| - \delta|k|} \alpha . \nonumber
\end{align}
Since the points in \eqref{eq-end points of diameter} are the endpoints of the diameter of the disk $ h^{-1}(\DD_1) $, the center is the arithmetic average of them. The radius is the half of the distance between two end points. Then the center is 
\begin{align*}
&\frac{1}{2} \left\{ \left(\frac{|k-1|}{|k-1| + \delta|k|} + \frac{|k-1|}{|k-1| - \delta|k|}\right) \left(-\frac{d}{c}\right) \right.  \\[0.2em]
&\qquad \qquad \qquad \qquad + \left. \left( \frac{\delta|k|}{|k-1| + \delta|k|} - \frac{\delta|k|}{|k-1| - \delta|k|} \right) \alpha \right\} \\[0.3em]
& = \frac{|k-1|^2}{|k-1|^2 - \delta^2|k|^2} \left(-\frac{d}{c}\right) + \frac{- \delta^2|k|^2}{|k-1|^2 - \delta^2|k|^2}\,\alpha
\end{align*}
Assume that $ \delta < \left| \frac{k-1}{k} \right| $. The radius is 
\begin{align*}
&\frac{1}{2} \left| \left(\frac{|k-1|}{|k-1| + \delta|k|} - \frac{|k-1|}{|k-1| - \delta|k|}\right) \left(-\frac{d}{c}\right) \right.  \\[0.2em]
&\qquad \qquad \qquad \qquad + \left. \left( \frac{\delta|k|}{|k-1| + \delta|k|} + \frac{\delta|k|}{|k-1| - \delta|k|} \right) \alpha\, \right| \\[0.3em]
& = \frac{\delta|k||k-1|}{|k-1|^2 - \delta^2|k|^2} \left| \frac{d}{c}+\alpha \right| \\
& = \frac{\delta|k|\sqrt{|k|}|k-1|}{|c|(|k-1|^2 + \delta^2|k|^2)} .
\end{align*}
\end{proof}

\begin{lemma} \label{lem-disk with radius epsilon}
The avoided region $ h^{-1} \left( \RR_{f}(1) \right) $ for small enough $ \delta > 0 $ contains the disk 
\begin{align*}
D_{\e_0} = \left\{ z \colon \left| z + \frac{d}{c} \right| < \e_0 \right\}
\end{align*}
where $ \e_0 \leq \frac{\delta|k|^2}{2|k-1|^2} \,| \alpha - \beta | $. 
\end{lemma}

\begin{proof}
The avoided region $ \RR_{f}(1) $ contains the disk $ \DD_1 = \left\{ w \colon \left| w- \frac{1}{k} \right| < \delta \right\} $ by the definition of the avoided region. Since $ h^{-1}\left(\frac{1}{k} \right) = -\frac{d}{c} $, the avoided region $ h^{-1} \left( \RR_{f}(1) \right) $ contains a disk of which center is $ -\frac{d}{c} $. Moreover, $ D_{\e_0} \subset h^{-1}(\DD_1) $ if and only if 
\begin{align} \label{eq-image under h-1 for extreme end}
\e_0 + \mathrm{dist} \big(\text{center} \ \text{of} \ \DD_1,\;  \text{center} \ \text{of} \ D_{\e_0} \big) \leq \text{radius} \ \text{of} \ \DD_1 .
\end{align}
Thus the distance between $ -\frac{d}{c} $ and the center of $ h^{-1}\left(\DD_1\right) $ is as follows 
\begin{align} \label{eq-first candidate for epsilon}
&\ \left| \frac{|k-1|^2}{|k-1|^2-\delta^2|k|^2} \left( -\frac{d}{c} \right) + \frac{-\delta^2|k|^2 \alpha}{|k-1|^2-\delta^2|k|^2} - \left( -\frac{d}{c} \right) \right| \nonumber \\[0.5em]
= &\ \left| \frac{\delta^2|k|^2 }{|k-1|^2-\delta^2|k|^2} \right| \left| \frac{d}{c} +\alpha \right| \nonumber \\[0.5em]
= &\ \frac{\delta^2|k|^2 \sqrt{|k|} }{|c|\big(|k-1|^2-\delta^2|k|^2 \big)} .
\end{align}
Recall that $ |c||\alpha-\beta| = \frac{|k-1|}{\sqrt{|k|}} $. The inequality in \eqref{eq-image under h-1 for extreme end} and the radius of $ h^{-1}\left(\DD_1\right) $ in Lemma \ref{lem-image of disk 1/k under h-1} implies that  
\begin{align} \label{eq-second candidate for epsilon}
\e_0 & \leq \frac{\delta|k|\sqrt{|k|}|k-1|}{|c|(|k-1|^2 + \delta^2|k|^2)} - \frac{\delta^2|k|^2 \sqrt{|k|} }{|c|\big(|k-1|^2-\delta^2|k|^2 \big)}  \nonumber \\[0.5em]
&= \frac{\delta|k|\sqrt{|k|}(|k-1| - \delta|k|)}{|c|(|k-1|^2 + \delta^2|k|^2)} \nonumber \\[0.5em]
&= \frac{\delta|k|\sqrt{|k|}}{|c|(|k-1| + \delta|k|)} \nonumber \\[0.5em]
&= \frac{\delta|k|^2}{(|k-1| + \delta|k|)|k-1|}\,|\alpha - \beta| .
\end{align}
Moreover, since we may choose $ \delta < \left| \frac{k-1}{k} \right| $, the disk $ D_{\e_0} $ is contained in $ h^{-1}(\DD_1) $ where $ \e_0 \leq \frac{\delta|k|^2}{2|k-1|^2}\,| \alpha - \beta | $. 
\end{proof}

\medskip

\noindent Denote $ h^{-1}(\RR_f(1)) $ by $ R_g(\infty) $, which is called the avoided region of $ g $ at $ \infty $. 

\medskip

\begin{corollary} \label{cor-upper bound of absolute value of g'}
For every $ z \in \C \setminus \RR_{g}(\infty) $, the following inequality holds
\begin{align*}
|g'(z)| \leq \frac{4|k-1|^2}{\delta^2|k|^3} .
\end{align*}
\end{corollary}

\begin{proof}
$ \RR_{g}(\infty) $ contains $ D_{\e_0} $ by Lemma \ref{lem-disk with radius epsilon}. Recall that $ g'(z) = \frac{1}{(cz + d)^2} $. In the region $ \C \setminus D_{\e_0} $, the inequality $ |cz + d| \geq |c| \e_o $ holds. The set inclusion $ D_{\e_0} \subset \RR_{g}(\infty) $ holds even if we choose the maximum value of $ \e_0 $. Thus the upper bound of $ |g'| $ is as follows
\begin{align*}
|g'(z)| = \frac{1}{|cz + d|^2} \leq \frac{1}{|c|^2 \e_0^2} = \frac{4|k-1|^4}{|c|^2 \delta^2 k^4 | \alpha - \beta |^2} 
\end{align*}
by Lemma \ref{lem-disk with radius epsilon}. Recall the equation
$ |c|| \alpha - \beta | = \frac{|k-1|}{\sqrt{|k|}} $. Hence,
\begin{align*}
\frac{4|k-1|^4}{|c|^2\delta^2 k^4 | \alpha - \beta |^2} = \frac{4|k-1|^4}{\delta^2|k|^4} \cdot \frac{|k|}{|k-1|^2} = \frac{4|k-1|^2}{\delta^2|k|^3} .
\end{align*}
\end{proof}


\noindent The following is the mean value inequality for holomorphic function.

\begin{lemma} \label{lem-mean value thoerem for holo functin}
Let $ g $ be the holomorphic function on the convex open set $ B $ in $ \C $. Suppose that $ \displaystyle \sup_{z \in B} |g'| < \infty $. Then for any two different points $ u $ and $ v $ in $ B $, we have
\begin{align*}
\left| \frac{g(u) - g(v)}{u-v} \right| \leq 2 \sup_{z \in B} |g'| .
\end{align*}
\end{lemma}

\begin{proof}
The complex mean value theorem implies that
\begin{align*}
\re(g'(p)) = \re \left( \frac{g(u) - g(v)}{u-v} \right) \quad \text{and} \quad \im(g'(q)) = \im \left( \frac{g(u) - g(v)}{u-v} \right) 
\end{align*}
where $ p $ and $ q $ are in the line segment between $ u $ and $ v $. Hence, the inequality 
$$ | \re(g'(p)) + i\,\im(g'(q)) | \leq | \re(g'(p))| + | \im(g'(q))| \leq 2 \sup_{z \in B} |g'| $$
completes the proof. 
\end{proof}

\medskip

\begin{proposition} \label{prop-stability on C minus S}
Let $ g $ be the loxodromic M\"obius map. For a given $ \e > 0 $, let a complex valued sequence $ \{ a_n \}_{n \in \N_0 } $ satisfies the inequality
$$
| a_{n+1} - g(a_n) | \leq \e
$$
for all $ n \in \N_0 $
. Suppose that $ a_0 \in \left({\C} \setminus B_R \right) \setminus \RR_{g}(\infty) $ where $ \RR_{g}(\infty) $ is the avoided region of $ g $ at $ \infty $ defined as $ h^{-1}(\RR_f(1)) $ for $ \frac{\left(\sqrt{|k|}+1 \right)^2}{\sqrt{|k|}(|k|-1)} $. Then there exists the sequence $ \{ b_n \}_{n \in \N_0 } $ defined as  
$$ 
b_{n+1} = g(b_n) 
$$
for each $ n =0,1,2, \ldots , N $ which satisfies that 
\begin{align*}
| a_N - b_N | &\leq \frac{M^N-1}{M-1}\, \e 
\end{align*}
where $ N $ is the uniformly escaping time from the region $ \left({\C} \setminus B_R \right) \setminus \RR_{g}(\infty) $ and $ M = \frac{8|k-1|^2}{\delta^2|k|^3} $. 
\end{proposition}

\begin{proof}
$ \frac{M}{2} $ is an upper bound of $ |g'| $ in $ \C \setminus \RR_{g}(\infty) $ by Corollary \ref{cor-upper bound of absolute value of g'}. The triangular inequality and Lemma \ref{lem-mean value thoerem for holo functin} implies that 
\begin{align} \label{eq-inequality for finite time stability}
| a_N - b_N | &\leq | a_N - g(a_{N-1}) | + | g(a_{N-1}) - g(b_{N-1}) | + | g(b_{N-1}) - b_N | \\
&\leq \e + M \, | a_{N-1} - b_{N-1} | 
\end{align}
where $ \displaystyle M \geq \sup_{z \in \C \setminus \RR_{g}(\infty)} 2|g'| $. The region $ B_R \subset S(1) $ by Corollary \ref{cor-h(partial S1)}. $ |g'(z)| > 1 $ if and only if $ z \in S(1) $. Thus $ M > 1 $. The inequality \eqref{eq-inequality for finite time stability} implies that 
\begin{align*}
| a_N - b_N | + \frac{\e}{M-1} \leq \, M \left( | a_{N-1} - b_{N-1} | + \frac{\e}{M-1} \right) .
\end{align*}
Then $ | a_N - b_N | $ is bounded above by the geometric sequence with rate $ M $, that is, 
\begin{align*}
| a_N - b_N | &\leq M^N\,\left( | a_{0} - b_{0} | + \frac{\e}{M-1} \right) - \frac{\e}{M-1} \\[0.2em]
&= M^N\, | a_{0} - b_{0} | + \frac{M^N-1}{M-1}\, \e .
\end{align*}
Hence, if we choose $ b_0 = a_0 $, then 
\begin{align*}
| a_N - b_N | &\leq \frac{M^N-1}{M-1}\, \e .
\end{align*}
\end{proof}

\noindent We show the Hyers-Ulam stability of loxodromic M\"obius map outside the avoided region.

\medskip

\begin{theorem} \label{thm-stabiliy of hyp Mobius transform}
Let $ g $ be a loxodromic M\"obius map. For a given $ \e > 0 $, let a complex valued sequence $ \{ a_n \}_{n \in \N_0 } $ satisfies the inequality
$$
| a_{n+1} - g(a_n) | \leq \e
$$
for all $ n \in \N_0 $. Suppose that a given point $ a_0 \in {\C} \setminus \RR_{g}(\infty) $ where $ \RR_{g}(\infty) $ is the avoided region defined in Section \ref{sec-Avoided region}. For a small enough number $ \e > 0 $, there exists the sequence $ \{ b_n \}_{n \in \N_0 } $ 
$$ 
b_{n+1} = g(b_n) 
$$
satisfies that $ |a_n - b_n | \leq H(\e) $ for all $ n \in \N_0 $ where the positive number $ H(\e) $ converges to zero as $ \e \rightarrow 0 $. 
\end{theorem}

\begin{proof}
Suppose first that $ a_0 \in B_R $. Then by Proposition \ref{prop-stability on D-R}, we have the inequality
\begin{align} \label{eq-stability first case}
|b_n - a_n| \leq \frac{1 - K^n}{1 - K} \,\e
\end{align}
for some $ K < 1 $ where $ b_0 = a_0 $. Assume also that $ a_0 \in \left({\C} \setminus B_R \right) \setminus \RR_{g}(\infty) $ and $ N \geq n $ where $ N $ is the escaping time. 
Then by Proposition \ref{prop-stability on C minus S}, 
\begin{align} \label{eq-stability second case}
|b_n - a_n| \leq \frac{M^n-1}{M-1}\, \e
\end{align}
where $ M = \frac{8|k-1|^2}{\delta^2|k|^3} $. Suppose that $ a_0 \in \left({\C} \setminus B_R \right) \setminus \RR_{g}(\infty) $ but $ n > N $ for the last case. Then we combine the first and second case as follows
\begin{align} \label{eq-stability third case}
|b_n - a_n| \leq \left( \frac{M^N-1}{M-1} + \frac{1 - K^{n-N}}{1 - K} \right)\e
\end{align}
where $ K $ and $ M $ are the numbers used in the inequality \eqref{eq-stability first case} and \eqref{eq-stability second case}. 
\end{proof}

\section*{Conclusion and final note}
Let $ \{b_n\}_{n\in \N_0} $ be the sequence satisfying the equation $ b_{n+1} = g(b_n) $ for all $ n \in \N_0 $ where $ g $ is the loxodromic M\"obius map on $ \hat{\C} $. Then the sequence has Hyers-Ulam stability on the exterior of some neighborhood of $ g^{-n}(\infty) $ for all $ n \in \N $. Since any hyperbolic M\"obius map is also loxodromic, the results and proofs in this paper could be also applied to the difference equation with the hyperbolic M\"obius map. All figures in the paper are based on the purely loxodromic M\"obius map
\begin{align*}
g(z) = \frac{1.64z - 25 + 11.07i}{0.04z + 0.27i} .
\end{align*}
Thus the fixed points of $ g $ are $ \alpha = 25 + 12i $ and $ \beta = 16 - 18.75i $. The trace of $ g $ is $ 1.64 + 0.27i $ and $ k = 1.5625 + 1.5i $.



\end{document}